\documentclass[12pt]{amsart}
\usepackage{tkz-euclide}
\usepackage{latexsym}
\usepackage{amssymb,amsmath,mathrsfs}
\usepackage{amsthm}
\usepackage{hyperref}
\usepackage{enumitem}
\usepackage[capitalise]{cleveref}

\allowdisplaybreaks 
\newtheorem{theorem}{Theorem}[section]
\newtheorem{definition}{Definition}[section]
\newtheorem{prop}[theorem]{Proposition}

\theoremstyle{definition}

\makeatletter
\def\th@remark{%
  \thm@headfont{\bfseries}%
  \normalfont 
  \thm@preskip\topsep \divide\thm@preskip\tw@
  \thm@postskip\thm@preskip
}
\makeatother

\theoremstyle{remark}

\newtheorem{lemma}[theorem]{Lemma}

\newcommand{\Bdot}{\mathring{B}}
\newcommand{\Balpha}{\mathring{B}^\alpha_{3,\infty}}

\newcommand{\Wdot}{\mathring{W}}
\newcommand{\Hdot}{\mathring{H}}

\newcommand{\half}{(-\overline{\Delta})^\frac{1}{2}}

\newcommand{\grad}{\overline{\nabla}}
\newcommand{\lap}{\overline{\Delta}}

\newcommand{\dnu}{\partial_\nu \Psi}

\def\XXint#1#2#3{{\setbox0=\hbox{$#1{#2#3}{\int}$ }
\vcenter{\hbox{$#2#3$ }}\kern-.6\wd0}}

\newcommand{\dive}{\operatorname{div}}
\newcommand{\curl}{\operatorname{curl}}

\newcommand{\supp}{\operatorname{supp}}
\newcommand{\riesz}{\mathcal{R}}
\newcommand{\Rthreeplus}{\mathbb{R}^3_+}
\newcommand{\Pgrad}{\mathbb{P}_{\nabla}}

\numberwithin{equation}{section}

\numberwithin{equation}{section}

\title[On the Weak Solutions to the 3D Inviscid Quasi-geostrophic System]{On the Weak Solutions to the 3D Inviscid Quasi-geostrophic System}

\author[Novack]{Matthew D. Novack}
\address[Matthew D. Novack]{\newline Department of Mathematics, \newline The University of Texas at Austin, Austin, TX 78712, USA}
\email{mnovack@math.utexas.edu}

\date{\today}


\begin{document}

\begin{abstract}
The purpose of this note is to study the weak solutions to the inviscid quasi-geostrophic system for initial data belonging to Lebesgue spaces. We give a global existence result as well as detail the connections between several different notions of weak solutions.  In addition, we give a condition under which the energy of the system is conserved.    
\end{abstract}
\maketitle \centerline{\date}

\section{introduction and main results}
We study the 3-D inviscid quasi-geostrophic system $(QG)$

\[
\left\{
       \begin{array}{@{}l@{\thinspace}l}
       \partial_t(\Delta \Psi) + \grad^\perp \Psi \cdot \grad (\Delta \Psi) = f_L  \hspace{.87in}  t>0,\hspace{.1in} z>0,\hspace{.1in} x=(x_1,x_2)\in\mathbb{R}^2 \hspace{.3in} \text{(QG)}_L  \\
       \partial_t(\dnu)+ \grad^\perp \Psi \cdot \grad(\dnu)= f_\nu \hspace{.84in}  t>0,\hspace{.1in} z=0,\hspace{.1in} x=(x_1,x_2)\in\mathbb{R}^2 \hspace{.3in} \text{(QG)}_\nu \\
       \end{array}  \right
.\]
supplied with an initial data $\Psi_0$. Here $$\Psi:\underset{(t,z,x)}{[0,\infty)\times\mathbb{R}_+\times\mathbb{R}^2} \underset{\rightarrow}{\rightarrow} \underset{\Psi(t,z,x)}{\mathbb{R}}$$ is the stream function for the geostrophic flow, and $f_L$ and $f_\nu$ are forcing terms.  
We use the notations
$$ \Delta \Psi = \partial_{x_1 x_1}\Psi +  \partial_{x_2 x_2}\Psi + \partial_{zz}\Psi , \hspace{.3in} \grad \Psi = (0, \partial_{x_1}\Psi, \partial_{x_2} \Psi )$$
$$ \grad^\perp \Psi =  (0, -\partial_{x_2}\Psi, \partial_{x_1} \Psi ),  \hspace{.3in} \dnu = -\partial_z \Psi(t,0,x).$$
The system is used to study stratified flows in which the Coriolis force is balanced with the pressure and serves as a model in simulations of large-scale atmospheric and oceanic circulation.    

  The purpose of this work is to study the existence and properties of various types of weak solutions to this system. We provide global existence results for initial data belonging to Lebesgue spaces and determine conditions under which a weak solution conserves the energy of the sytem.  Much mathematical research has been focused on this system and its variants.  Beale and Bourgeois \cite{bb} and Desjardins and Grenier \cite{dg} derived the system from physical principles.  Puel and Vasseur \cite{pv} first proved the global existence of weak solutions in the case of $L^2$ initial data, using a projection operator to reformulate the problem.  In the case when $\Delta \Psi_0 \equiv 0$ and there are no forcing terms, $\Psi(t)$ remains harmonic for all time $t$.  Then using that $\half$ is the Dirichlet-to-Neumann operator for $\Rthreeplus$, we set for each time $t\geq 0$
$$\theta := \dnu = \half \Psi , \qquad  u:= \grad^\perp\Psi = \left( 0, -\mathcal{R}_2 \half \Psi, \mathcal{R}_1 \half \Psi \right) = \mathcal{R}^\perp \theta, $$
where $\mathcal{R}_1, \mathcal{R}_2$ are the Riesz transforms in $\mathbb{R}^2$. Then $(QG)$ reduces to the well-studied inviscid surface quasi-geostrophic equation, which can be written as
$$ \partial_t \theta +u \cdot \grad \theta =0. $$
For the sake of consistency and to keep in mind the connection to the 3D model, we shall always treat $\grad, \grad^\perp$, and $\mathcal{R}^\perp$ as vectors with three components and zero first component.  SQG has received considerable attention due to its similarities with the important systems of fluid mechanics (see Constantin, Majda, and Tabak \cite{cmt}, Garner, Held, Pierrehumber, and Swanson \cite{MR1312238}, among others).  Weak solutions were constructed in $L^2$ by Resnick \cite{Resnick}.  Marchand \cite{Marchand} first gave a proof of the existence of global weak solutions when the initial data is not in $L^2$ but rather $L^p$ for any $p>\frac{4}{3}$. When critical dissipation is added to the transport equation for $\dnu$ in (QG), global regularity was established in \cite{novackvasseur} using the De Giorgi technique in combination with a bootstrapping argument and an appropriate Beale-Kato-Majda criterion. Surface quasi-geostrophic flows on bounded domains have been considered by Constantin and Ignatova \cite{ci}, \cite{ci2}, Constantin and Nguyen \cite{cn}, \cite{cn2} and Nguyen \cite{nguyen} using the spectral Riesz transform. Global existence of weak solutions for inviscid SQG is shown by Constantin and Nguyen \cite{cn} and for a generalized SQG model by Nguyen \cite{nguyen}. In \cite{boundeddomains}, an appropriate boundary condition is derived and global weak solutions to the 3D model posed on a cylindrical domain are constructed. 
  
\subsection{The Reformulated System}
A crucial tool in our analysis will be a reformulation of $(QG)$. We draw inspiration from Puel and Vasseur \cite{pv}, who used a reformulation to obtain their global existence result. The physical system as written is analogous to the vorticity form of the Euler equations with an additional boundary condition.  However, one may consider the following reformulation, in which $\curl(Q)$ acts as a Lagrange multiplier similar to the gradient of the pressure in the Euler equations: 

\[
\left\{
       \begin{array}{@{}l@{\thinspace}l}
       \partial_t(\nabla \Psi) + \grad^\perp \Psi : \grad (\nabla \Psi) = \curl Q + \nabla F  \hspace{.455in} z>0  \\
       \curl Q \cdot \nu = 0 , \hspace{.1in} \partial_\nu F = f_\nu \hspace{1.56in} z=0\\
       \Delta F = f_L  \hspace{2.675in} z>0. \hspace{.311in} \text{(rQG)} \\
       \end{array}  \right
.\]

Formally, taking the divergence of $(rQG)$ gives $(QG)_L$, and taking the trace gives $(QG)_\nu$.  To obtain $(rQG)$ from $(QG)$, one must invert the divergence operator coupled with a Neumann boundary condition.  While providing a link between the two formulations will be an important part of our analysis (see \cref{connectionstheorem}), let us proceed from the perspective of $(rQG)$ for the time being.  Following Puel and Vasseur \cite{pv}, we define the notion of weak solutions to $(rQG)$.

\begin{definition}[{\textbf{Weak Solutions to $\mathbf{(rQG)}$}}]\label{weakrqg}
Let $T,R$ be fixed, $\phi\in C^\infty(\mathbb{R}^4)$ compactly supported in $(-T,T)\times(-R,R)^3$, and $F$ be such that $\Delta F = f_L$, $\partial_\nu F = f_\nu$.  A weak solution $\Psi$ to (rQG) with forcing $f_\nu$, $f_L$ on $(0,T)\times\Rthreeplus$ must satisfy 

\begin{align*}
-\int_0^T \int_0^\infty &\int_{\mathbb{R}^2} \left(\left( \partial_t \nabla \phi + \grad^\perp\Psi:\grad\nabla\phi \right) \cdot \nabla \Psi  + \nabla \phi \cdot \nabla F \right )\,dx\,dz\,dt\\
&= \int_0^\infty \int_{\mathbb{R}^2} \nabla\phi(0,z,x)\cdot \nabla\Psi(0,z,x)\,dx\,dz
\end{align*}
for all $R, \phi$.  For the weak formulation to make sense, we require $ \nabla\Psi,\grad^\perp\Psi\otimes\nabla \Psi \in L_{loc}^1([0,T]\times\Rthreeplus)$.
\end{definition}

We remark that the definition of weak solutions contains no information about $\curl(Q)$.  Indeed, the choice of test functions formally encodes the fact that inverting the divergence operator is unique only up to the curl of a vector field.  

\subsection{Statement of Main Results}
We begin with the global existence of weak solutions to $(rQG)$.  
\begin{theorem}\label{main}
Suppose that $p\in(\frac{4}{3},\infty]$ and $q\in(\frac{6}{5}, 3]$.
Let $\theta \in L^p(\mathbb{R}^2)$, $\omega \in L^q(\Rthreeplus)$, $f_L \in L^1\left([0,T];L^\frac{6}{5}\cap L^q(\Rthreeplus)\right)$, and $f_\nu \in L^1\left([0,T];L^\frac{4}{3} \cap L^p(\mathbb{R}^2)\right)$ for all $T>0$. When $p=\infty$ we additionally require a finite $p'$ such that $\theta\in L^{p'}(\mathbb{R}^2)$, and when $q=3$ we additionally require a $q'\in (\frac{6}{5},3)$ such that $\omega \in L^{q'}(\Rthreeplus)$. Then there exists a global weak solution $\nabla \Psi$ on $(0,\infty)\times \Rthreeplus$ to $(rQG)$ with forcing $f_\nu$, $f_L$ such that $\Delta \Psi |_{t=0}= \omega$ and $\dnu|_{t=0}=\theta$.  In addition, there exists a constant $C$ such that for all $T>0$, $\Psi$ satisfies the following bound: 
\begin{align*}
\|\nabla \Psi &\|_{L^\infty\left([0,T];L^\frac{3p}{2}(\mathbb{R}_+^3) + L^\frac{3q}{3-q}(\Rthreeplus)\right)} +  \|\Delta \Psi\|_{L^\infty\left([0,T];L^q(\Rthreeplus)\right)} + \|\dnu\|_{L^\infty\left([0,T];L^p(\mathbb{R}^2)\right)} \\
&\leq C\left( \|\omega\|_{L^q} + \|\theta\|_{L^p} + \|f_L\|_{L^1\left([0,T];L^q(\Rthreeplus)\right)} + \|f_\nu\|_{L^1\left([0,T];L^p(\mathbb{R}^2)\right)}\right) . \\ 
\end{align*}
\end{theorem}

Let us give a simple explanation for the restrictions on $p$ and $q$. In order for the nonlinear term $\grad\cdot(\grad^\perp \Psi \otimes \nabla\Psi )$ to be well-defined as a distribution from integration by parts, we need $\nabla \Psi \in L^2(\Rthreeplus)$ (at least locally).  If $\Delta\Psi_0 \in L^\frac{6}{5}(\Rthreeplus)$ and $\dnu_0\in L^\frac{4}{3}(\mathbb{R}^2)$, then solving the elliptic boundary value problem gives $\nabla \Psi_0 \in L^2(\Rthreeplus)$, hence the restrictions on $q$ and $p$.  If $q=3$ or $p=\infty$, the corresponding Lebesgue norm on $\nabla \Psi$ is actually the standard $BMO$ norm in the space of functions of bounded mean oscillation; for simplicity's sake we employ this abbreviation. The additional assumptions on $\theta$ when $p=\infty$ and $\omega$ when $q=3$ are technical requirements which are necessary to handle the decay at infinity of functions defined in $\Rthreeplus$. The solutions we construct are obtained by taking a weak limit of smooth solutions to a regularized system.  Global smooth solutions for the regularized system are constructed following \cite{novackvasseur}. We refer to the preliminaries for a precise statement of the result we shall use, and the appendix for a brief description of the techniques. 

The following theorem addresses the conservation of the energy $\|\nabla\Psi (t) \|_{L^2(\Rthreeplus)}$ in the case of no forcing.  Here $\Balpha(\mathbb{R}^2)$ is the usual homogeneous Besov space.
\begin{theorem}\label{conservation}
Let $\nabla\Psi$ be a weak solution to $(rQG)$ with no forcing such that
$$ \nabla\Psi \in C\left([0,T);L^2(\Rthreeplus)\right) \cap L^3\left([0,T)\times[0,\infty); \Balpha(\mathbb{R}^2)\right) $$
for some $\alpha > \frac{1}{3}$.  Then $\|\nabla\Psi(t)\|_{L^2(\Rthreeplus)} = \|\nabla\Psi_0\|_{L^2(\Rthreeplus)}$ for $t\in[0,T)$.
\end{theorem}

In the case $\Delta\Psi_0 \equiv 0$, the system reduces to SQG and one has the equality 
$$ \|\nabla\Psi(t)\|_{L^2(\Rthreeplus)} = \|\dnu(t)\|_{\Hdot^{-\frac{1}{2}}(\mathbb{R}^2)}. $$
The quantity $\|\dnu(t)\|_{\Hdot^{-\frac{1}{2}}(\mathbb{R}^2)}$ is actually the Hamiltonian of the system in this case; see Resnick \cite{Resnick} or Buckmaster, Shkoller, Vicol \cite{2016arXiv161000676B}.  Buckmaster, Shkoller, and Vicol  provide a proof of the non-uniqueness of weak solutions below a certain regularity threshold.  Conversely, Isett and Vicol \cite{Isett2015} prove that when $\dnu\in L^3_{t,x}$, the Hamiltonian is conserved.

\subsection{Relations between (QG) and (rQG)}

It is interesting to consider whether weak solutions to $(rQG)$ might be weak solutions to $(QG)$, and vice versa. In this section we address this question, therein justifying our use of the reformulated system. We define two classes of weak solutions to $(QG)$; the first is the more standard notion of weak solution, while the second incorporates the Calder\'{o}n commutator used in the existence proofs of Marchand \cite{Marchand} and Resnick \cite{Resnick}. 

\begin{definition}[{\textbf{Weak Solutions to $\mathbf{(QG)}$}}]\label{weakqg}
Let $T,R$ be fixed, $\phi\in C^\infty(\mathbb{R}^4)$ compactly supported in $(-T,T)\times(0,R)\times(-R,R)^2$, and $\bar{\phi}\in C^\infty(\mathbb{R}^3)$ compactly supported in $(-T,T)\times(-R,R)^2$. A weak solution $\Psi$ to $(QG)$ on $(0,T)\times\Rthreeplus$ with forcing $f_\nu$, $f_L$ must satisfy

\begin{align}
-\int_0^T \int_0^\infty &\int_{\mathbb{R}^2} \left(\left( \partial_t \phi + \grad^\perp \Psi \cdot \grad \phi  \right) \Delta \Psi + \phi f_L  \right) \,dx\,dz\,dt \nonumber \\
&= \int_0^\infty \int_{\mathbb{R}^2} \phi(0,z,x)\Delta\Psi(0,z,x) \,dx\,dz \label{qg1}
\end{align}

and

\begin{align}
-\int_0^T &\int_{\mathbb{R}^2} \left(\left(  \partial_t \bar{\phi} + \grad^\perp \Psi(t,0,x) \cdot \grad \bar{\phi}\right) \dnu(t,x) + \bar{\phi}f_\nu  \right) \,dx \,dt \nonumber \\
&= \int_{\mathbb{R}^2} \bar{\phi}(0,x)\dnu(0,x)\,dx \label{qg2}
\end{align}
for all $R, \phi, \bar{\phi}$.  For the weak formulation to make sense, we require $\Delta\Psi, \grad^\perp\Psi \Delta\Psi \in L_{loc}^1([0,T]\times\Rthreeplus)$ and $\dnu,\grad^\perp\Psi \dnu \in L_{loc}^1([0,T]\times\mathbb{R}^2)$.
\end{definition}

For functions of two variables, $\bar{\Lambda}\theta = \sqrt[]{-\bar{\Delta}}(\theta)$ and $\bar{\Lambda}^{-1}$ is the corresponding inverse operator. In addition, $$\riesz^\perp \theta = (0,-\riesz_2 \theta, \riesz_1 \theta)$$
is the rotated vector of Riesz transforms with zero first component as usual.  The commutator $[A,B]$ of two operators is $AB - BA$. In the following definition, we use the commutator result of Marchand \cite{Marchand} to define a notion of weak solution for $(QG)$ for low levels of integrability.  Marchand's results concerning boundedness and convergence of the commutator are stated in the preliminaries.  For the sake of brevity we suppress for now issues concerning the frequency support of $\dnu$; these are also addressed in the preliminaries. 

\begin{definition}[{\textbf{Weak Solutions to $\mathbf{(QG)}$ with Commutator}}]\label{weakqgc}
Let $T,R$ be fixed, $\phi\in C^\infty(\mathbb{R}^4)$ compactly supported in $(-T,T)\times(0,R)\times(-R,R)^2$, and $\bar{\phi}\in C^\infty(\mathbb{R}^3)$ compactly supported in $(-T,T)\times(-R,R)^2$. Let $\Psi:[0,T)\times\Rthreeplus\rightarrow\mathbb{R}$ be given and $\Psi_1$ and $\Psi_2$ be defined for all $t\in[0,T)$ by the boundary value problems
\[
 \left\{
       \begin{array}{@{}l@{\thinspace}l}
       \Delta \Psi_{1} = 0 \\
       \dnu_{1} = \dnu
       \end{array}  \right.
\hspace{.3in} 
 \left\{
       \begin{array}{@{}l@{\thinspace}l}
       \Delta \Psi_{2} = \Delta \Psi \\
       \dnu_{2} = 0.
       \end{array}  \right.
\]
 We define $\left(\grad^\perp\Psi(t,0,x) \dnu(t,x)\right)_C$ as a distribution by (and use the notation $(\cdot)_C$ to specify that we are using the commutator formulation) 
\begin{align*}
\int_0^T \int_{\mathbb{R}^2}  &\left( \grad^\perp \Psi(t,0,x) \dnu(t,x) \right)_C \cdot\grad \bar{\phi}  \,dx\,dt := \frac{1}{2} \int_0^T \int_{\mathbb{R}^2} \left( \riesz^\perp (\dnu_1) \right) \cdot \left( [\bar{\Lambda}, \grad\bar{\phi}] (\bar{\Lambda}^{-1}\dnu_1) \right) \,dx\,dt \\
&\qquad\qquad\qquad+ \int_0^T \int_{\mathbb{R}^2}  \left( \grad^\perp \Psi_2(t,0,x) \dnu_1(t,x) \right)\cdot \grad\bar{\phi} \,dx\,dt\\
&=\frac{1}{2} \int_0^T \int_{\mathbb{R}^2} \bigg{(}0, -\mathcal{R}_2 (\partial_\nu \Psi_1), \mathcal{R}_1 (\partial_\nu \Psi_1) \bigg{)} \cdot \bigg{(} \bar{\Lambda}\big{(}\grad\phi \bar{\Lambda}^{-1} ( \partial_\nu \Psi_1 ) \big{)} - \grad\phi \partial_\nu \Psi_1 \bigg{)} \,dx\,dt \\
&\qquad\qquad\qquad + \int_0^T \int_{\mathbb{R}^2}  \left( \grad^\perp \Psi_2(t,0,x) \dnu_1(t,x) \right)\cdot \grad\bar{\phi} \,dx\,dt
\end{align*}
and say that $\Psi$ is a weak solution to $(QG)$ with commutator on $(0,T)\times\Rthreeplus$ with forcing $f_\nu$, $f_L$ if 
\begin{align*}
-\int_0^T \int_0^\infty &\int_{\mathbb{R}^2} \left(\left( \partial_t \phi + \grad^\perp \Psi \cdot \grad \phi  \right) \Delta \Psi + \phi f_L  \right) \,dx\,dz\,dt \nonumber \\
&= \int_0^\infty \int_{\mathbb{R}^2} \phi(0,z,x)\Delta\Psi(0,z,x) \,dx\,dz \label{qg1}
\end{align*}
and
\begin{align*}
-\int_0^T &\int_{\mathbb{R}^2} \left(  \partial_t \bar{\phi} \dnu + \left(\grad^\perp \Psi\dnu\right)_C  \cdot \grad \bar{\phi} + \bar{\phi}f_\nu  \right) \,dx \,dt \nonumber \\
&= \int_{\mathbb{R}^2} \bar{\phi}(0,x)\dnu(0,x)\,dx 
\end{align*}
for all $T, R, \phi, \bar{\phi}$.  For the weak formulation to make sense, we require $\dnu(t)\in L^p(\mathbb{R}^2)$ for all time $t$ and some $p\in(\frac{4}{3},2]$ and $\grad^\perp\Psi_2 \dnu \in L_{loc}^1([0,T]\times\mathbb{R}^2)$.
\end{definition}
See the preliminaries for Marchand's convergence result regarding the commutator and other details.  

We now connect the weak solutions of \cref{weakrqg}, \cref{weakqg}, and \cref{weakqgc}.

\begin{theorem}\label{connectionstheorem}
\begin{enumerate}
\item\label{connectionspartone}  Assume that $\Delta\Psi \in L^\infty\left([0,T);L^q(\Rthreeplus)\right)$ for $q\in[\frac{3}{2},3]$ and $\dnu\in L^\infty\left([0,T);L^p(\mathbb{R}^2)\right)$ for $p\in[2,\infty]$. Then $\nabla\Psi$ satisfies \cref{weakrqg} if and only if $\nabla\Psi$ satisfies \cref{weakqg}. 
\item Assume that $\Delta\Psi \in L^\infty\left([0,T);L^q(\Rthreeplus)\right)$ for $q\in[\frac{3}{2},3]$ and $\dnu\in L^\infty\left([0,T);L^p(\mathbb{R}^2)\right)$ for $p\in(\frac{4}{3},2]$. Assume in addition that $$ p\geq \frac{2q}{3(q-1)}. $$ Then $\nabla\Psi$ satisfies \cref{weakrqg} if and only if $\nabla\Psi$ satisfies \cref{weakqgc}.
\item Assume that $\Delta\Psi \in L^\infty\left([0,T);L^q(\Rthreeplus)\right)$ for $q\in[\frac{3}{2},3]$ and $\dnu\in L^\infty\left([0,T);L^p\cap L^r(\mathbb{R}^2)\right)$ for $p\in(\frac{4}{3},2]$, $r\in[2,\infty]$. Then $\nabla\Psi$ satisfies \cref{weakqg} if and only if $\nabla\Psi$ satisfies \cref{weakqgc}.
\end{enumerate}

\end{theorem}

\cref{connectionstheorem} complements the existence result in \cref{main}.  Indeed, imposing that the initial data $\nabla\Psi_0$, $\Delta \Psi_0$, and $\dnu_0$ all belong to $L^2$, then we recover the result of Puel and Vasseur \cite{pv}.  Imposing $\Delta\Psi_0 \equiv 0$ and $\dnu_0\in L^p(\mathbb{
R}^2)$, we recover the result of Marchand \cite{Marchand}. 

It is interesting to note that if the initial data satisfies $\Delta \Psi_0 \in L^\frac{6}{5}(\Rthreeplus)$ and $\partial_\nu \Psi_0 \equiv 0$ to remove the boundary condition, trace theory would give $\grad^\perp\Psi_0|_{z=0} \in L^\frac{4}{3}(\mathbb{R}^2)$ (see \cref{trace}), corresponding precisely to the lower limit of integrability in the proof of Marchand. Conversely, imposing that $\Delta\Psi_0 \equiv 0$ and $\dnu_0 \in L^\frac{4}{3}(\mathbb{R}^2)$ to eliminate the transport equation for $z>0$, \cref{laxmilgram2} ensures that $\nabla\Psi_0 \in L^2(\Rthreeplus)$. In addition, one can see from the proof of \cref{connectionstheorem} that 
$$ p\geq \frac{2q}{3(q-1)} $$
is the minimum integrability needed to define the nonlinear terms in both $(QG)_L$ and $(QG)_\nu$. Thus, the conditions on $p$ and $q$ correspond in a natural way and appear to be the sharpest possible afforded by the structure of the system.  Furthermore, our analysis combines the reformulation $(rQG)$ of Vasseur and Puel and the commutator of Marchand. In conjunction with the correspondence between the conditions on $p$ and $q$, this naturally connects the two approaches.  

\section{Preliminaries}
\subsection{Definitions and Previous Results}

We collect several definitions and known results as well as state and prove the elliptic estimates necessary for the proof of our main theorems. We begin with an existence result for a regularized (QG) system.  The proof follows the method from \cite{novackvasseur}.  We provide a short summary of the techniques and their implementation here in the appendix. The weak solutions we build to (QG) will be limits of the regularized system described below.  While it would suffice to build "approximate" solutions by adding a stronger diffusion (the method employed by Puel and Vasseur to build approximate solutions \cite{pv}) or by regularizing the velocity fields slightly (the method employed in \cite{boundeddomains}), the following approach has the advantage of providing a brief summary of techniques used in previous work on a closely related system.  

\begin{theorem}[{\textbf{Regularized System}}]\label{regularizedsolutions}
Consider the regularized system $(QG)_\epsilon$
\[
\left\{
       \begin{array}{@{}l@{\thinspace}l}
       \partial_t(\Delta \Psi_\epsilon) + \grad^\perp \Psi_\epsilon \cdot \grad (\Delta \Psi_\epsilon) = f_{L,\epsilon}  \hspace{1.65in}  t>0,\hspace{.1in} z>0,\hspace{.1in} x=(x_1,x_2)\in\mathbb{R}^2 \hspace{.3in} \\
       \partial_t(\dnu_\epsilon)+ \grad^\perp \Psi_\epsilon \cdot \grad(\dnu_\epsilon)= f_{\nu,\epsilon}-\epsilon\half(\dnu_\epsilon) \hspace{.44in}  t>0,\hspace{.1in} z=0,\hspace{.1in} x=(x_1,x_2)\in\mathbb{R}^2 \hspace{.3in}  \\
       \end{array}  \right
.\]
supplied with initial data $\Delta\Psi_{0,\epsilon}$, $\dnu_{0,\epsilon}$ which are $C^\infty$ and compactly supported.  Suppose that $f_{L,\epsilon}\in L^1\left([0,T];L^1 \cap L^q(\Rthreeplus)\right) \cap L^\infty([0,T];C^k(\Rthreeplus))$, $f_{\nu,\epsilon}\in L^1\left([0,T];L^1\cap L^p(\mathbb{R}^2)\right) \cap L^\infty([0,T];C^k(\mathbb{R}^2))$ for all $T>0$, $k\in\mathbb{N}$ and that for each time, $f_{L,\epsilon}$ and $f_{\nu,\epsilon}$ have spatial support contained in $[-\frac{5}{\epsilon},\frac{5}{\epsilon}]^3$ and $[-\frac{5}{\epsilon},\frac{5}{\epsilon}]^2$, respectively.  Then there exists a unique, global in time classical solution $\nabla\Psi_\epsilon$ and a constant $C$ independent of $\epsilon$ such that $\nabla\Psi_\epsilon$ satisfies the energy estimates for $t\in[0,T]$ 
\begin{enumerate}
\item $\| \Delta \Psi_\epsilon(t)\|_{L^q} \leq C (\|f_{L,\epsilon}\|_{L^1([0,T];L^q)} + \| \Delta\Psi_{0,\epsilon} \|_{L^q}) $
\item $\| \dnu_\epsilon(t)\|_{L^p} \leq C (\|f_{\nu,\epsilon}\|_{L^1([0,T];L^p)} + \| \dnu_{0,\epsilon} \|_{L^p}) $
\item $\| \nabla \Psi_\epsilon(t)\|_{L^\frac{3q}{3-q}+L^\frac{3p}{2}} \leq C (\|f_{L,\epsilon}\|_{L^1([0,T];L^q)} + \| \Delta\Psi_{0,\epsilon} \|_{L^q} + \|f_{\nu,\epsilon}\|_{L^1([0,T];L^p)} + \| \dnu_{0,\epsilon} \|_{L^p} ) $
\end{enumerate}
\end{theorem}
An outline of the proof following \cite{novackvasseur} is contained in the appendix. Let us now state results of Marchand \cite{Marchand}.
\begin{lemma}[{\textbf{Calder\'{o}n Commutator}}]\label{Marchand} 
\begin{enumerate}
\item For $f\in L^p(\mathbb{R}^2)$, $p\in(\frac{4}{3},2]$, and $\phi \in \mathcal{D}((0,T)\times \mathbb{R}^2) $, $\grad\cdot(f \riesz^\perp f)$ is defined as a distribution by 
$$ \langle \phi, \grad\cdot(f \riesz^\perp f) \rangle := \frac{1}{2} \int_{\mathbb{R}^2} \left(\riesz^\perp f \right)\cdot \left( [\bar{\Lambda}, \grad\phi] \left(\bar{\Lambda}^{-1}f\right)\right). $$
If $f\in L^2(\mathbb{R}^2)$ is such that $\hat{f}$ is zero in a neighborhood of the origin, then 
$$ \int_{\mathbb{R}^2} \left( f \riesz^\perp f \right) \cdot \grad \phi = -\frac{1}{2} \int_{\mathbb{R}^2} \left(\riesz^\perp f \right)\cdot \left( [\bar{\Lambda}, \grad\phi] \left(\bar{\Lambda}^{-1}f\right)\right). $$
\item Let $p\in(\frac{4}{3},\infty]$ and $\{\theta_{\epsilon}(t,x)\}_{\epsilon>0} \subset L^\infty([0,T]; L^p(\mathbb{R}^2))$ be a sequence of functions such that $\theta_\epsilon$ converges weakly-* to $\theta(t,x)\in L^\infty([0,T]; L^p(\mathbb{R}^2))$, $T$ fixed.  Then the following holds in the sense of distributions: 
$$ \lim_{\epsilon\rightarrow 0}{\grad \cdot (\theta_\epsilon \riesz^\perp \theta_\epsilon) = \grad \cdot (\theta \riesz^\perp \theta) }.  $$
Here it is understood that for $p\geq 2$, $\grad\cdot(\theta_\epsilon \riesz^\perp \theta_\epsilon)$ is defined by integration by parts, whereas for $p\leq 2$, we use the commutator.
\end{enumerate}
\end{lemma}

Decomposing an arbitrary function using Littlewood-Paley projections allows one to use the commutator only for the high-frequency piece.  To avoid cumbersome Besov space notations, we suppress these details and will write 
$$ \riesz^\perp\theta [\bar{\Lambda}, \grad\phi] \left(\bar{\Lambda}^{-1}\theta\right) $$ 
for any $L^p$ function with $p\in(\frac{4}{3},2]$.  We refer the reader to Marchand \cite{Marchand} for further details and proofs.  

For the proof of \cref{conservation}, we shall need several identities, definitions, and notations concerning Littlewood-Paley decompositions and Besov spaces. The homogeneous Besov spaces $\Balpha(\mathbb{R}^2)$ are defined via the usual bi-infinite sequence of homogeneous Littlewood-Paley decompositions (per the text of Bahouri, Chemin, and Danchin \cite{bcd}).  Here $\{\gamma_\epsilon \}_{\epsilon>0}$ is a sequence of compactly supported, radially symmetric approximate identities.  For a function $u:\mathbb{R}^2 \rightarrow \mathbb{R}^n$, we define $u^\epsilon := u \ast \gamma_\epsilon$.
\begin{prop}\label{lpstuff}
\begin{enumerate}
\item\label{dos} For $L^1_{loc}$ functions $f$ and $g$, $$ \int_{\mathbb{R}^2} g (f^\epsilon)^\epsilon = \int_{\mathbb{R}^2} g^\epsilon f^\epsilon $$
\item\label{tres} The following commutator identity holds:
\begin{align*}
(f\cdot g)^\epsilon(x) - &f^\epsilon g^\epsilon(x) =\\
&\int_{\mathbb{R}^2} \int_{\mathbb{R}^2} \left( f(x-\bar{x}) - f(x) \right)\gamma_\epsilon(\bar{x}) \left( g(x-\bar{x}) - g(x-x') \right) \gamma_\epsilon(x') \,dx'\,d\bar{x}
\end{align*}
\item\label{cuatro}  For $\alpha\in(0,1)$ and $u \in \Balpha(\mathbb{R}^2)$, there exists $C$ independent of $u$ such that for all $|y|>0$,
$$ \|u(\cdot - y) - u(\cdot)\|_{L^3(\mathbb{R}^2)} \leq C y^\alpha \|u\|_{\Balpha(\mathbb{R}^2)} $$
and
$$ \|\grad u^\epsilon\|_{L^3(\mathbb{R}^2)} \leq C \epsilon^{\alpha-1} \|u\|_{\Balpha(\mathbb{R}^2)}. $$
\end{enumerate}
\end{prop}

\begin{proof}
(1) follows immediately from a change of variables and the radial symmetry of the mollifier.  For (2), we can write 
\begin{align*}
(f\cdot g)^\epsilon(x) - &f^\epsilon g^\epsilon(x) = \int_{\mathbb{R}^2} f(x-\bar{x})g(x-\bar{x})\gamma_\epsilon(\bar{x})\,d\bar{x} \\ 
&\qquad \qquad - \int_{\mathbb{R}^2} f(x-\bar{x})\gamma_\epsilon(\bar{x}) \,d\bar{x} \cdot \int_{\mathbb{R}^2} g(x-x')\gamma_\epsilon(x') \,dx' \\
&= \int_{\mathbb{R}^2} f(x-\bar{x}) \gamma_\epsilon(\bar{x}) \left( g(x-\bar{x}) - \int_{\mathbb{R}^2} g(x-x')\gamma_\epsilon(x') \,dx' \right)  \,d\bar{x}\\
&= \int_{\mathbb{R}^2} \int_{\mathbb{R}^2} \left( f(x-\bar{x}) - f(x) \right)\gamma_\epsilon(\bar{x}) \left( g(x-\bar{x}) - g(x-x') \right) \gamma_\epsilon(x') \,dx'\,d\bar{x}
\end{align*}
Statements and proofs of (3) can be found in the text of Bahouri, Chemin, and Danchin \cite{bcd}.
\end{proof}

The homogeneous Sobolev spaces are defined by
$$ \Wdot^{1,r}(\Rthreeplus) := \{ u \in \mathcal{D}'(\Rthreeplus) | \nabla u \in L^r(\Rthreeplus) \} $$
with norm 
$$ \|u\|_{\Wdot^{1,r}(\Rthreeplus)} = \|\nabla u \|_{L^r(\Rthreeplus)}. $$
Strictly speaking, for the norm to be well-defined and for the following inequality to hold, we consider equivalence classes of distributions which differ by an additive constant.  Let us recall the classical Escobar inequality for the half-space $\Rthreeplus$ \cite{Esc}. 

\begin{lemma}\label{trace}
Suppose that $q\in[1,3)$, and $u\in \Wdot^{1,q}(\Rthreeplus)$.  Then 
$$ \|u|_{z=0}\|_{L^\frac{2q}{3-q}(\mathbb{R}^2)} \leq C(q) \|u\|_{\Wdot^{1,q}(\Rthreeplus)} $$
\end{lemma}

\subsection{Elliptic Estimates}

We now specify the appropriate Lebesgue spaces and obtain the corresponding bounds for the solution to the Poisson problem with Neumann boundary data in the upper half space.  While the results are standard, we include proofs for the sake of completeness.  We also include a technical lemma which will be useful in the proof of \cref{main}.
\begin{lemma}\label{laxmilgram1}
Given $f \in L^q(\Rthreeplus)$ for $q\in(1,3]$, there exists a unique $u\in\Wdot^{1,\frac{3q}{3-q}}(\Rthreeplus)$ ($\nabla u \in BMO$ if $q=3$) such that 
\[
 \left\{
       \begin{array}{@{}l@{\thinspace}l}
       -\Delta u = f  \qquad z>0 \\
       \partial_\nu u = 0 \hspace{.45in} z=0
       \end{array}  \right.
       \]
with $$ \|\nabla u \|_{L^\frac{3q}{3-q}(\Rthreeplus)} \leq C(q) \| f \|_{L^q(\Rthreeplus)}, \qquad q<3 $$ 
or
$$  \|\nabla u \|_{BMO(\Rthreeplus)} \leq C(q) \| f \|_{L^q(\Rthreeplus)}, \qquad q=3 . $$
\end{lemma}

\begin{proof}
Let us begin with the case $q=3$.  Applying the operator whose symbol is $\frac{i\xi}{|\xi|^2}$ (we ignore constants coming from the Fourier transform) to 
\begin{displaymath}
   f_E(z,x) = \left\{
     \begin{array}{lr}
       f(z,x) &  z>0\\
       0 &  z\leq 0
     \end{array}
   \right.
\end{displaymath} 
gives a curl free vector field in $BMO(\mathbb{R}^3)$ which is in fact the gradient of a function $u_E$ (see, for example, Temam \cite{Temam2001}).  Then applying the same operator to 
\begin{displaymath}
   f_{E,r}(z,x) = \left\{
     \begin{array}{lr}
       0 &  z>0\\
       f(-z,x) &  z\leq 0
     \end{array}
   \right.
\end{displaymath} 
yields a vector field in $BMO(\mathbb{R}^3)$ which is again the gradient of a function $u_{E,r}$.  Putting $u=u_E + u_{E,r}$, it is clear that $-\Delta u =f $ in $\Rthreeplus$ and $$\partial_\nu u = \partial_\nu u_E + \partial_\nu u_{E,r} = \partial_\nu u_E - \partial_\nu u_E = 0.$$
The bound follows from the boundedness of the multiplier operator from $L^3(\mathbb{R}^3)$ to $BMO(\mathbb{R}^3)$.

We use the generalized Lax-Milgram theorem for Banach spaces (see for example Theorem 8.10 in the text of Arbogast and Bona \cite{arbogastandbona}) to show the existence as well as the bound for $q<3$.  Define $X:= \Wdot^{1,\frac{a}{a-1}}(\Rthreeplus)$ and $ Y:= \Wdot^{1,a}(\Rthreeplus)$.  Define $B:X \times Y \rightarrow \mathbb{R}$ by 
$$ B(u,v) = \int_{\Rthreeplus} \nabla u \cdot \nabla v   $$ 
and $F(v): Y \rightarrow \mathbb{R}$ by 
$$ F(v) = \int_{\Rthreeplus} vf  .$$
Choosing $a=\frac{3q}{4q-3}$ gives that $v\in L^\frac{q}{q-1}(\Rthreeplus)$ by Sobolev embedding, and thus $F$ is well-defined and continuous.  Continuity of $B$ follows from H\"{o}lder's inequality.  We must show $B$ to be non-degenerate, i.e. 
$$ \sup_{u\in X} B(u,v) >0  \hspace{.2in} \forall v\in Y $$
and coercive, i.e.
$$ \inf_{\substack{u\in X \\ \|u\|=1} }  \sup_{\substack{v\in Y \\ \|v\|=1}} B(u,v)  \geq \gamma >0.$$
To show coercivity, we begin by fixing $u\in X$ with $\|u\|_{\Wdot^{1, \frac{a}{a-1}}}=1$.  The ideal choice for $\nabla v$ would be $\nabla u |\nabla u|^{\frac{a}{a-1}-2}$.  Of course, this may not be the gradient of a function.  Therefore, let us define the operator $\Pgrad$ for Schwartz vector fields $s:\mathbb{R}^3 \rightarrow \mathbb{R}^3$ by 
\begin{align*}
\widehat{\Pgrad (s)}(\xi) &= \left( \frac{\langle \hat{s}(\xi), \xi \rangle}{|\xi|^2} \xi \right) \\
&= \left(  \sum_{i=1}^3 \hat{s_i}(\xi) \frac{\xi_i \xi_j}{|\xi|^2}\right), \hspace{.4in} j=1,2,3.\\
\end{align*}
Recalling that the symbol for the $j^{th}$ Riesz transform $\mathcal{R}_j$ is $-\frac{i\xi_j}{|\xi|}$, $\Pgrad$ is a linear combination of compositions of Riesz transforms.  We then extend $\Pgrad$ by density as a bounded operator from $\left(L^r(\mathbb{R}^3)\right)^3$ to itself for all $r\in(1,\infty)$.  In addition, for $s_1$ scalar valued, $s_2$ vector valued  Schwarz functions, examining the symbol of $\Pgrad$ shows that
$$ \langle \nabla s_1, \Pgrad s_2\rangle = \langle \nabla s_1, s_2\rangle. $$
Continuity of the operator ensures that this property remains true for vector fields in $X$ and $Y$.  We define $ u_E(z,x) = u(|z|,x)$ to be the symmetric extension of $u$ over the plane $z=0$.  With this definition, 
\begin{equation}\label{symu1}
\partial_z u_E(z,x) = -\partial_z u_E(-z,x)
\end{equation}
and 
\begin{equation}\label{symu2}
\grad u_E(z,x) = \grad u_E(-z,x).
\end{equation}
We apply $\Pgrad$ to the extended vector field $\nabla u_E |\nabla u_E|^{\frac{a}{a-1}-2}$.  Using the symmetry and antisymmetry of the Riesz transforms and $\nabla u_E |\nabla u_E|^{\frac{a}{a-1}-2}$ with respect to reflection over the plane $z=0$, it is simple to check that 
\begin{equation}\label{sym1}
\partial_z\Pgrad \left( \nabla u_E |\nabla u_E|^{\frac{a}{a-1}-2}\right)(z,x) = -\partial_z \Pgrad \left(\nabla u_E |\nabla u_E|^{\frac{a}{a-1}-2}\right)(-z,x) 
\end{equation}
and 
\begin{equation}\label{sym2}
\grad \Pgrad \left(\nabla u_E |\nabla u_E|^{\frac{a}{a-1}-2}\right)(z,x) = \grad \Pgrad \left(\nabla u_E |\nabla u_E|^{\frac{a}{a-1}-2}\right)(-z,x).
\end{equation}

We set $\nabla v(z,x) = \frac{1}{\|\Pgrad\|}\Pgrad (\nabla u_E |\nabla u_E|^{\frac{a}{a-1}-2}) |_{z\geq 0}$.  By direct computation, $\|\nabla v \|_{L^a(\Rthreeplus)} \leq 1$, and using \eqref{symu1}, \eqref{symu2}, \eqref{sym1}, and \eqref{sym2} gives that
\begin{align*}
\int_{\Rthreeplus} \nabla u \cdot \nabla v &= \frac{1}{2\|\Pgrad\|} \int_{\mathbb{R}^3} \nabla u_E \cdot \Pgrad(\nabla u_E |\nabla u_E|^{\frac{a}{a-1}-2})\\
&= \frac{1}{2\|\Pgrad\|} \int_{\mathbb{R}^3} |\nabla u_E|^\frac{a}{a-1}\\
&= \frac{1}{\|\Pgrad\|}.\\
\end{align*}
Thus the coercivity is shown with $\gamma = \frac{1}{\|\Pgrad\|}$.  Non-degeneracy follows from switching $u$ and $v$ and repeating the argument.  Therefore, the conditions of Lax-Milgram are met, and we have the existence of a solution $u$ to the  variational problem, as well as the gradient bound on $u$ in terms of $f$.  Then, taking $v$ to be compactly supported in $\Rthreeplus$ shows that $-\Delta u = f$ in the sense of distributions.  Now, taking $v\in \mathcal{D}(\mathbb{R}^3)$ shows that $\partial_\nu u$ is well defined as a distribution by $$ \int_{\Rthreeplus} \nabla u \cdot \nabla v + v \Delta u =: \int_{\mathbb{R}^2} v \partial_\nu u  $$
and is equal to zero.  
\end{proof}

For the following lemma we use the space $$\Wdot_{\Delta}^{1,p}(\Rthreeplus) := \{ u \in \Wdot^{1,p}(\Rthreeplus) | \Delta u = 0 \textit{ in } \mathcal{D}'(\Rthreeplus) \}$$
with norm 
$$ \|u\|_{\Wdot_{\Delta}^{1,p}(\Rthreeplus)} = \|\nabla u \|_{L^p(\Rthreeplus)} $$

\begin{lemma}\label{laxmilgram2}
Given $g \in  L^p(\mathbb{R}^2)$ for $p\in (1, \infty]$, there exists $u \in \Wdot^{1, \frac{3p}{2}}_{\Delta}(\Rthreeplus)$ solving 
\[
 \left\{
       \begin{array}{@{}l@{\thinspace}l}
       \Delta u = 0 \qquad z>0\\
      \partial_\nu u = g \hspace{.3in} z=0
       \end{array}  \right.
       \]
with 
$$ \|\nabla u\|_{L^\frac{3p}{2}(\Rthreeplus)} \leq C(p) \|g\|_{L^p(\mathbb{R}^2)}, \qquad p<\infty $$
or
$$ \|\nabla u\|_{BMO(\Rthreeplus)} \leq C(p) \|g\|_{L^p(\mathbb{R}^2)}, \qquad p=\infty.  $$
\end{lemma}

\begin{proof}
Let us begin with the case $p=\infty$.  Applying the Poisson kernel $\mathcal{P}(z,x)$ to $g(x)$ gives a harmonic function in $\Rthreeplus$.  Considering the vector field 
$$ v(z,x) = -\left( \mathcal{P}(z,\cdot)\ast g(\cdot)(x), \riesz_1\mathcal{P}(z,\cdot)\ast g(\cdot)(x) , \riesz_2\mathcal{P}(z,\cdot)\ast g(\cdot)(x) \right) ,$$
it is clear that $v$ is curl free and is thus the gradient of a harmonic function $u$ with $\partial_\nu u = g$.  The bound follows from noting that the Riesz transforms are bounded from $L^\infty(\mathbb{R}^2)$ to $BMO(\mathbb{R}^2)$ and $\|\mathcal{P}(z,\cdot)\ast g(\cdot)(x)\|_{L^\infty(\mathbb{R}^2)} \leq \|g(x)\|_{L^\infty(\mathbb{R}^2)}$ for all $z$.

We use again the Lax-Milgram theorem for $p<\infty$.  Define $X:= \Wdot_{\Delta}^{1,\frac{3p}{2}}(\Rthreeplus)$ and $Y:= \Wdot^{1, \frac{3p}{3p-2}}(\Rthreeplus)$.  Let $B:X\times Y \rightarrow \mathbb{R}$ be defined by 
$$ B(u,v) = \int_{\Rthreeplus} \nabla u \cdot \nabla v $$
and $F:Y \rightarrow \mathbb{R}$ be defined by 
$$ F(v) = \int_{\mathbb{R}^2} v|_{z=0} g. $$
By \cref{trace}, we have that $v|_{z=0} \in L^\frac{p}{p-1}(\mathbb{R}^2)$, and therefore $F$ is well-defined and continuous.  Continuity of $B$ follows from H\"{o}lder's inequality.  As before, we are tasked with showing the coercivity and non-degeneracy of $B$.  Making use of the $\Pgrad$ operator, the details follow as in the previous lemma and are omitted.  The existence of $u$ and the gradient bound in terms of $g$ are provided by the Lax-Milgram theorem.  Taking $v$ compactly supported in $\Rthreeplus$ shows that indeed $\Delta u = 0$.  We then again have that $\partial_\nu u$ is well-defined as a distribution from integration by parts and satisfies $\partial_\nu u = g$.
\end{proof}

The following lemma regarding the strong convergence of solutions to the Laplace equation with Neumann boundary data shall be useful in the proof of \cref{main}. Of particular importance is the fact that the convergence holds up to the boundary $z=0$ when $p>\frac{4}{3}$, providing a stronger result than interior regularity estimates.

\begin{lemma}\label{strongconvergence}
Let $\{g_\epsilon\}_{\epsilon>0}$ be a bounded sequence of functions in $L^p(\mathbb{R}^2)$ for $p>\frac{4}{3}$.  Let $u_\epsilon(z,x):\Rthreeplus\rightarrow\mathbb{R}$ be the solution to
\[
 \left\{
       \begin{array}{@{}l@{\thinspace}l}
       \Delta u_\epsilon= 0 \hspace{.395in} z>0 \\
       \partial_\nu u_\epsilon = g_\epsilon \qquad z=0
       \end{array}  \right.
       \]
Then there exists $u$ such that up to a subsequence, $\nabla u_\epsilon$ converges strongly to $\nabla u$ in $L^2\left((0,R)\times B_R(0)\right)$ for all $R>0$.   
\end{lemma}

\begin{proof}
Fix $R>0$. We first extract a subsequence which we shall continue to call $\{g_\epsilon\}$ in an abuse of notation that converges weakly-* to $g$ in $L^p(\mathbb{R}^2)$. Applying \cref{laxmilgram2} to $g_\epsilon$ gives that $ u_\epsilon$ converges weakly-* to $ u$ in $\Wdot^{1,\frac{3p}{2}}(\Rthreeplus)$, where $u$ solves the Laplace equation with Neumann data $g$.  Because $p>\frac{4}{3}$, we have that $\frac{3p}{2}>2$, and therefore $\{\nabla u_\epsilon\}$ is a weakly-* convergent sequence in $L^2\left((0,R)\times B_R(0)\right)$.  Note that $\nabla u_\epsilon$ is harmonic for all $\epsilon$ and thus the harmonic extension satisfies for fixed $z$ that
$$ \nabla u _\epsilon (z,x) = -\left( \mathcal{P}(z,\cdot)\ast g_\epsilon(\cdot)(x), \riesz_1\mathcal{P}(z,\cdot)\ast g_\epsilon(\cdot)(x) , \riesz_2\mathcal{P}(z,\cdot)\ast g_\epsilon(\cdot)(x) \right), $$
and similarly for $u$. Furthermore, by the smoothness and decay at infinity of the Poisson kernel $\mathcal{P}$ away from the boundary $z=0$, $\mathcal{P}\ast g_\epsilon$ and $\mathcal{P} \ast g$ belong to $W^{k,\frac{3p}{2}}\left((\delta,\infty)\times\mathbb{R}^2\right)$ for any $k\in\mathbb{N}$ and fixed $\delta>0$. Taking the Riesz transform shows that the same holds for $\riesz (\mathcal{P} \ast g_\epsilon)$ and $\riesz(\mathcal{P} \ast g)$. Then by the Rellich-Kondrachov theorem, $\nabla u_\epsilon$ converges strongly to $\nabla u$ up to a subsequence in $L^2((\delta,R)\times B_R(0))$. Thus fixing $0<\delta<R$, we can write
\begin{align*}
\limsup_{\epsilon\rightarrow 0} \int_0^R \int_{B_R(0)} &|\nabla u_\epsilon (z,x) - \nabla u(z,x)|^2 \,dx\,dz\\ 
&= \limsup_{\epsilon\rightarrow 0} \bigg{(}  \int_0^\delta \int_{B_R(0)} |\nabla u_\epsilon (z,x) - \nabla u(z,x)|^2 \,dx\,dz\\
&\qquad\qquad + \int_\delta^R \int_{B_R(0)} |\nabla u_\epsilon (z,x) - \nabla u(z,x)|^2 \,dx\,dz \bigg{)}\\
&\leq \sup_{\epsilon>0} \int_0^\delta \int_{B_R(0)} |\nabla u_\epsilon (z,x) - \nabla u(z,x)|^2 \,dx\,dz \\
&\leq \sup_{\epsilon>0} \|\nabla u_\epsilon - \nabla u\|^2_{L^\frac{3p}{2}(\Rthreeplus)} \| \mathcal{X}_{\{[0,\delta]\times B_{R}(0) \}} \|_{L^\frac{3p}{3p-4}(\Rthreeplus)}\\
&\leq C \left( R^2 \delta \right)^\frac{3p-4}{3p}
\end{align*}
after applying the uniform bound on $g_\epsilon$ in $L^p(\mathbb{R}^2)$ and H\"{o}lder's inequality.  Considering that $p>\frac{4}{3}$ and $R$ is fixed, the final expression approaches zero as $\delta$ decreases to zero.  Diagonalizing the subsequence $u_\epsilon$ over $R\in\mathbb{N}$ finishes the proof.   
\end{proof}

We close this section by recalling the Hodge decomposition from Vasseur and Puel \cite{pv}, with an additional higher regularity bound which will be useful.

\begin{lemma}[\textbf{Hodge Decomposition}]\label{hodge}
Let $v \in H^3(\Rthreeplus)$. Then there exists a unique decomposition
$$ v = \nabla w + \curl u, \qquad \nabla w, \curl u \in L^2(\Rthreeplus) $$  satisfying 
$$ \int_{\Rthreeplus} \nabla \phi \cdot v = \int_{\Rthreeplus} \nabla \phi \cdot \nabla w$$
for any $\nabla \phi \in L^2(\Rthreeplus)$. In addition, we have the following higher regularity bound:
$$\| \nabla w \|_{H^3(\Rthreeplus)} \lesssim \| v \|_{H^3(\Rthreeplus)}.$$
Finally, if the support of $v$ is compact, then $\nabla^2 w \in L^{1+\delta}(\Rthreeplus)$ for any $\delta>0$.
\end{lemma}

\begin{proof}
Proposition 3.2 from Vasseur and Puel's work \cite{pv} shows the existence of the unique decomposition
$$ v = \nabla w + \curl u $$
given $v\in L^2(\Rthreeplus)$ with the desired orthogonality condition in $L^2(\Rthreeplus)$.  Since $v\in H^3(\Rthreeplus)$, in particular $v\in L^2(\Rthreeplus)$, and we can apply their result to conclude the existence of a unique $\nabla w$, $\curl u$ belonging to $L^2(\Rthreeplus)$ and satisfying
$$ \int_{\Rthreeplus} \nabla \phi \cdot v = \int_{\Rthreeplus} \nabla \phi \cdot \nabla w. $$
We now show the higher regularity bound $ \|\nabla w\|_{H^3} \lesssim \| v \|_{H^3(\Rthreeplus)}. $ The proof utilizes the classical Nirenberg difference quotients.

Let the difference quotient operator $T_h$ for the chosen direction $x = (x_1, x_2) \in \mathbb{R}^2$ be defined by 
$$ T_h(f)(z',x') := \frac{f(z',x'+hx) - f(z',x')}{h} $$
and let $\partial_x$ be the corresponding partial differential operator. Then all quantities in the following expression are well-defined and we can write
\begin{align*}
\int_{\Rthreeplus} \nabla (T_h w) \cdot \nabla (T_h w) &= -\int_{\Rthreeplus} \nabla (T_{-h} T_h  w) \cdot \nabla w \\
&= -\int_{\Rthreeplus} \nabla (T_{-h} T_h w) \cdot v\\
&= \int_{\Rthreeplus} \nabla (T_h w) \cdot (T_h v).
\end{align*}
Applying Cauchy's inequality, we conclude
$$ \| \nabla (T_h w) \|_{L^2(\Rthreeplus)} \lesssim \| v \|_{H^1(\Rthreeplus)}$$
with a bound uniform in $h$.  Passing to a limit as $h\rightarrow 0$ shows that then $\nabla (\partial_x w) \in L^2(\Rthreeplus)$; to show that $\partial_{zz} w \in L^2(\Rthreeplus)$, we observe that $\partial_{zz}w = \Delta w - \lap w = \nabla \cdot v - \lap w$.  Therefore 
$$ \| \nabla w \|_{H^1} \lesssim \| v \|_{H^1}. $$

For the $H^2$ bound, we can first write that
\begin{align*}
\int_{\Rthreeplus} \nabla (T_h T_h w) \cdot \nabla (T_h T_h w) &= \int_{\Rthreeplus} \nabla (T_{-h} T_{-h} T_h T_h  w) \cdot \nabla w \\
&= \int_{\Rthreeplus} \nabla (T_{-h} T_{-h} T_{h} T_h w) \cdot v\\
&= \int_{\Rthreeplus} \nabla (T_h T_h w) \cdot (T_h T_h v).
\end{align*}
From here we conclude as before that $ \|  \nabla (\partial_{xx} w) \|_{L^2} \lesssim \| v \|_{H^2} $. Since $\partial_{zzx} = \partial_x (\Delta-\lap)$ and $\Delta w = \nabla \cdot v$, we have that $\partial_{zzx} w \in L^2$.  In addition, since $\partial_{zzz} = \partial_z(\Delta - \lap)$, we have that $\partial_{zzz}w \in L^2$, and therefore
$$ \| \nabla w \|_{H^2} \lesssim \| v \|_{H^2}. $$
For the $H^3$ bound, we can argue as above to conclude that $\| \nabla (\partial_{xxx} w) \| \lesssim \| v \|_{H^3}$. The full bound then follows from the identities
$$\partial_{zzxx} = \partial_{xx}(\Delta - \lap), \qquad \partial_{zzzx} = \partial_{zx}(\Delta - \lap), \qquad \partial_{zzzz} = \partial_{zz}(\Delta - \lap) .$$

It remains to show the $L^{1+\delta}$ bound on $\nabla^2 w$ in the case that $\supp v$ is compact. When $\supp v$ is compact, $\Delta w = \nabla\cdot v$ and $\partial_\nu w = v\cdot\nu$ are therefore both compactly supported and $L^1$. By the elliptic bounds in \cref{laxmilgram1} and \cref{laxmilgram2}, $\nabla w \in L^{\frac{3}{2}+\delta}(\Rthreeplus)$ for any $\delta>0$.  By Sobolev embedding, $w \in L^{3+\delta}(\Rthreeplus)$ for any $\delta>0$. Classical estimates for harmonic functions then give that for $\alpha \in \Rthreeplus$ sufficiently large (far outside the support of $v$), 
\begin{align*}
| \nabla^2 w(\alpha) | &\lesssim \frac{1}{|\alpha|^5} \| w \|_{L^1(B(\alpha, \frac{|\alpha|}{2}))}\\
&\lesssim \frac{1}{|\alpha|^5} \|w\|_{L^{3+\delta}(\Rthreeplus)} |\alpha|^{3 \left(\frac{2+\delta}{3+\delta}\right)}.
\end{align*}
Thus $\nabla^2 w$ decays at a rate of $\frac{1}{|\alpha|^{3-\delta}}$ in $\Rthreeplus$ for any $\delta>0$, showing that $\nabla^2 w \in L^{1+\delta}(\Rthreeplus)$ for any $\delta>0$.
\end{proof}

\section{Proof of Theorem 1.1}

We now have the estimates necessary for the proof of the main theorem.  Here we assume that $p$ and $q$ satisfy the assumptions of \cref{main}.

\begin{proof}[Proof of \cref{main}]
Let $\{\gamma_\epsilon\}_{\epsilon>0}$ be a sequence of approximate identities compactly supported in $B_\epsilon(0)$ in $\mathbb{R}^2$ and $\{\Gamma_\epsilon\}_{\epsilon>0}$ a sequence of approximate identities compactly supported in $B_\epsilon(0)$ in $\mathbb{R}^3$.  We define truncated versions of the initial data and forcing by
$$ \omega_{T_\epsilon} =\omega \mathcal{X}_{\{|\omega|<\frac{1}{\epsilon}, |(z,x)| < \frac{1}{\epsilon}\}} , \hspace{.2in} \theta_{T_\epsilon} =\theta \mathcal{X}_{\{|\theta|<\frac{1}{\epsilon}, |x| < \frac{1}{\epsilon}\}}  ,$$
with $f_{L, T_\epsilon}(t)$ and $f_{\nu, T_\epsilon}(t)$ defined analogously for each time $t\geq 0$.  Then we regularize by putting
$$ \omega_\epsilon = \Gamma_\epsilon \ast \omega_{T_\epsilon} , \hspace{.2in} \theta_\epsilon =  \gamma_\epsilon\ast\theta_{T_\epsilon}, \hspace{.2in} f_{L,\epsilon}(t) = \Gamma_\epsilon \ast f_{L,T_\epsilon}(t) , \hspace{.2in} f_{\nu,\epsilon}(t) = \gamma_\epsilon \ast f_{\nu, T_\epsilon}(t),$$
ensuring that $\omega_\epsilon$, $\theta_\epsilon$, $f_{L,\epsilon}(t)$, and $f_{\nu, \epsilon}(t)$ are compactly supported, $C^\infty$ functions in space for each $t\geq 0$. Setting $\Delta\Psi_{0,\epsilon}:=\omega_\epsilon$ and $\dnu_{0,\epsilon}:=\theta_\epsilon$ shows that the assumptions of \cref{regularizedsolutions} are satisfied. Therefore there exists a classical solution $\nabla\Psi_\epsilon$ to
\[
\left\{
       \begin{array}{@{}l@{\thinspace}l}
       \partial_t(\Delta \Psi_\epsilon) + \grad^\perp \Psi_\epsilon \cdot \grad (\Delta \Psi_\epsilon) = f_{L,\epsilon}  \hspace{1.65in}  t>0,\hspace{.1in} z>0,\hspace{.1in} x=(x_1,x_2)\in\mathbb{R}^2 \hspace{.3in} \\
       \partial_t(\dnu_\epsilon)+ \grad^\perp \Psi_\epsilon \cdot \grad(\dnu_\epsilon)= f_{\nu,\epsilon}-\epsilon\half(\dnu_\epsilon) \hspace{.45in}  t>0,\hspace{.1in} z=0,\hspace{.1in} x=(x_1,x_2)\in\mathbb{R}^2 \hspace{.3in}  \\
       \end{array}  \right
.\]
Define $F_\epsilon:[0,\infty)\times\Rthreeplus\rightarrow\mathbb{R}$ for all time by
\[
 \left\{
       \begin{array}{@{}l@{\thinspace}l}
       \Delta F_\epsilon= f_{L,\epsilon} \hspace{1.505in} z>0 \\
       \partial_\nu F_\epsilon = f_{\nu,\epsilon}-\epsilon \half\dnu_\epsilon \qquad z=0
       \end{array}  \right.
       \]
Integrating by parts with a smooth test function $\phi(t,z,x)$ with compact spacial support in $\overline{\Rthreeplus}$, we have the following equalities:
\begin{align*}
-\int_0^T \int_0^\infty \int_{\mathbb{R}^2} &\left(\left( \partial_t \nabla \phi + \grad^\perp\Psi_\epsilon : \grad \nabla \phi \right)\cdot\nabla\Psi_\epsilon + \nabla\phi\cdot\nabla F_\epsilon \right) \,dx\,dz\,dt  \nonumber\\
& \qquad =\int_0^T \int_0^\infty \int_{\mathbb{R}^2} \left(\left( \partial_t \phi + \grad^\perp \Psi_\epsilon \cdot \grad \phi \right)\Delta\Psi_\epsilon  + \phi \Delta F_\epsilon \right)  \,dx\,dz\,dt \nonumber\\
&\qquad\qquad - \int_0^T \int_{\mathbb{R}^2} \left(\left( \partial_t \phi + \grad^\perp \Psi_\epsilon \cdot \grad \phi \right) \dnu_\epsilon + \phi \partial_\nu F_\epsilon \right) \,dx\,dt
\end{align*}
and
\begin{align*}
\int_0^\infty \int_{\mathbb{R}^2} \nabla \phi(0,z,x) \cdot \nabla\Psi_\epsilon(0,z,x) \,dx\,dz &= -\int_0^\infty \int_{\mathbb{R}^2} \phi(0,z,x) \Delta \Psi_\epsilon(0,z,x) \,dx\,dz   \nonumber\\
& \qquad+  \int_{\mathbb{R}^2} \phi(0,0,x)\dnu_\epsilon(0,0,x) \,dx
\end{align*}
Using that $\nabla\Psi_\epsilon$ is a solution to the regularized system, the right hand sides of the above equalities are in fact equal, and therefore the left hand sides are equal as well, i.e.
\begin{align}\label{weakrqgsmoothdata}
- \int_0^T \int_0^\infty &\int_{\mathbb{R}^2} \left(\left( \partial_t \nabla \phi+ \grad^\perp\Psi_\epsilon:\grad\nabla\phi \right) \cdot \nabla \Psi_\epsilon  + \nabla \phi \cdot \nabla F_\epsilon \right )\,dx\,dz\,dt \nonumber \\
 &= \int_0^\infty \int_{\mathbb{R}^2} \nabla\phi(0,z,x)\cdot\nabla\Psi_\epsilon(0,z,x)\,dx\,dz
\end{align}
If the support of $\phi$ is not compact but $\Delta \phi$ and $\partial_\nu \phi$ are compactly supported, we claim the equality \eqref{weakrqgsmoothdata} still holds under approximation by smooth functions. By \cref{hodge}, $\nabla^2 \phi \in L^{1+\epsilon}\cap L^\infty(\Rthreeplus)$, ensuring that 
$$ \grad^\perp\Psi_\epsilon : \grad \nabla \phi \cdot \nabla \Psi_\epsilon $$
is bounded using H\"{o}lder's inequality and we can pass to the limit from a sequence of compactly supported functions. In addition, $\nabla \phi \in L^2\cap L^\infty(\Rthreeplus)$, ensuring that $\nabla \phi \cdot \nabla F_\epsilon$ is well-defined by the assumptions on the integrability of $f_\nu$ and $f_L$. 

To pass to the limit in \eqref{weakrqgsmoothdata}, we use \cref{regularizedsolutions} to detail the spaces in which $\{\Psi_\epsilon\}$ is pre-compact.  Throughout, $T>0$ is fixed, and weak-$\ast$ convergence is abbreviated simply as weak convergence. We decompose $\Psi_\epsilon(t) = \Psi_{\epsilon,1}(t) + \Psi_{\epsilon,2}(t)$ as follows:  
\[
 \left\{
       \begin{array}{@{}l@{\thinspace}l}
       \Delta \Psi_{\epsilon,1} = 0 \\
       \dnu_{\epsilon,1} = \dnu_\epsilon 
       \end{array}  \right.
\hspace{.3in} 
 \left\{
       \begin{array}{@{}l@{\thinspace}l}
       \Delta \Psi_{\epsilon,2} = \Delta \Psi_\epsilon \\
       \dnu_{\epsilon,2} = 0.
       \end{array}  \right.
\]

\begin{enumerate}
\item\label{itm:one} By \cref{regularizedsolutions}(2), $\{\dnu_{\epsilon,1}\}$ is bounded in $L^\infty([0,T];L^p(\mathbb{R}^2))$ and we can pass to a weakly convergent subsequence.  
\item\label{itm:three} By \cref{regularizedsolutions}(1), $ \{ \Psi_{\epsilon,2} \} $ is bounded in $L^\infty([0,T];\Wdot^{2,q}(\mathbb{R}_+^3))$ and we can pass to a weakly convergent subsequence. 
\end{enumerate}

Given the weak convergence of $\dnu_\epsilon=\dnu_{\epsilon,1}$ and $\Delta \Psi_\epsilon=\Delta\Psi_{\epsilon,2}$, we will show that up to a subsequence, $\nabla\Psi_\epsilon=\nabla\Psi_{\epsilon,1}+\nabla\Psi_{\epsilon,2}$ converges strongly in $L^\infty([0,T];L^2((0,R)\times B_0(R)))$ for any $R$.  To prove this, we use the Aubin-Lions lemma \cite{a} (as do Puel and Vasseur \cite{pv}); note also that here is where require $p>\frac{4}{3}$ and $q>\frac{6}{5}$. We break the argument into steps. The first step specifies the Banach space in which $\{\nabla\Psi_{\epsilon}\}$ is bounded.  The second step specifies the Banach space in which $\{\partial_t \nabla \Psi_\epsilon\}$ is bounded.  The last step specifies the relationship between these Banach spaces and $L^\infty([0,T];L^2((0,R)\times{B_0(R)}))$, justifying the use of the Aubin-Lions lemma.

$Step$ $One:$ Let $\nabla h\in C_c^\infty(\overline{\Rthreeplus})$.  Define
$$ \|\nabla h\|_{B_1} := \|\Delta h\|_{L^q(\Rthreeplus)} + \|\partial_\nu h\|_{L^p(\mathbb{R}^2)}. $$
Define the space of gradients
$$ B_1 := \textnormal{cl}(C_c^\infty(\overline{\Rthreeplus})) $$
to be the closure of $C_c^\infty(\overline{\Rthreeplus})$ gradients of functions in the upper half space with respect to the norm $\|\cdot\|_{B_1}$. 
Approximating $\nabla\Psi_\epsilon(t)$ by gradients of smooth, compactly supported functions shows that $\nabla\Psi_\epsilon(t)\in B_1$ for each $t\in[0,T]$, and thus $\{\nabla\Psi_\epsilon\}\subset L^\infty([0,T];B_1)$ is a bounded sequence.

$Step$ $Two:$ The distributional time derivative $\partial_t \nabla \Psi_\epsilon$ (in the sense of the Aubin-Lions lemma) is defined by the equality 
\begin{align}\label{distributionaltimederivative}
\langle \partial_t \nabla \Psi_\epsilon, h \rangle := -\int_0^T \nabla\Psi_\epsilon(t) h'(t) \,dt
\end{align}
for all $h\in C_c^\infty(0,T)$. Define the Sobolev space $V$ to be the closure of $C_c^\infty\left([0,R)\times{B_0(R)}\right)$ vector fields under the usual $H^3(\Rthreeplus)$ norm, and set $B_{-1}$ to be the dual space $V^*$. To show that $\{\partial_t \nabla \Psi_\epsilon\}$ is a bounded sequence in $L^\infty([0,T];B_{-1})$, we test \eqref{distributionaltimederivative} against a vector field $v\in V$. By \cref{hodge}, we have that
$$ -\int_0^T \int_{\Rthreeplus} \nabla\Psi_\epsilon(t) h'(t) v(z,x) \,dt \,dz\,dx = -\int_0^T \int_{\Rthreeplus} \nabla\Psi_\epsilon(t) h'(t) \nabla w(z,x) \,dt \,dz\,dx$$
Using again \cref{hodge}, we have that $\nabla w\in H^3(\Rthreeplus)$ and $\nabla^2 w \in L^{1+\delta}(\Rthreeplus)$ for any $\delta>0$. The assumptions on the integrability of $f_\nu$ and $f_L$ in \cref{main} ensure that $\nabla F_\epsilon \in L_t^1(L^2(\Rthreeplus))$, and therefore $\nabla F_\epsilon \cdot \nabla w$ is well-defined and integrable independently of $\epsilon$. The assumptions on the integrability of $\theta$, $\omega$, $f_L$, and $f_\nu$ in \cref{main} and the estimates in \cref{regularizedsolutions} ensure that $\nabla\Psi_\epsilon(t)$ always belongs to $L^\frac{3p}{2}+L^\frac{3q}{3-q}(\Rthreeplus)$ for some $p<\infty$, $q<3$ uniformly in $t$ and $\epsilon$, and therefore
$$ \grad^\perp\Psi_\epsilon : \grad \nabla w \cdot \nabla \Psi_\epsilon$$
is well-defined and integrable uniformly in $\epsilon$ by H\"{o}lder's inequality. Thus all terms in equality \eqref{weakrqgsmoothdata} are well-defined and bounded uniformly in $\epsilon$ for the test function $\phi=h(t)w(z,x)$ . In conclusion, we have that $\{\partial_t \nabla\Psi_\epsilon\}$ is a bounded sequence in $L^\infty([0,T];B_{-1})$.

$Step$ $Three:$ The inclusion of $\left(L^2((0,R)\times{B_0(R)})\right)^3$ into $B_{-1}$ is continuous. We now show that the inclusion of $B_1$ into $\left(L^2((0,R)\times{B_0(R)})\right)^3$ is compact.  Given a bounded sequence $\{\nabla z_n\}\subset B_1$, decompose as follows:  
\[
 \left\{
       \begin{array}{@{}l@{\thinspace}l}
       \Delta z_{n,1} = 0 \\
       \partial_\nu z_{n,1} = \partial_\nu z_n 
       \end{array}  \right.
\hspace{.3in} 
 \left\{
       \begin{array}{@{}l@{\thinspace}l}
       \Delta z_{n,2} = \Delta z_n \\
       \partial_\nu z_{n,2} = 0.
       \end{array}  \right.
\]
Since $p>\frac{4}{3}$, we can apply \cref{strongconvergence} to $\{\nabla z_{n,1}\}$, yielding strong convergence of a subsequence in $\left(L^2((0,R)\times{B_0(R)})\right)^3$.  Using that $z_{n,2}\in \Wdot^{2,q}(\Rthreeplus)$ and $q>\frac{6}{5}$, the Rellich-Kondrachov theorem yields in addition strong convergence of a subsequence $\{\nabla z_{n,2}\}$ in $\left(L^2((0,R)\times{B_0(R)})\right)^3$. Summing $z_n = z_{n,1}+z_{n,2}$ gives that $\nabla z_n$ converges strongly in $\left(L^2((0,R)\times{B_0(R)})\right)^3$, and therefore $B_1$ embeds compactly in $\left(L^2((0,R)\times{B_0(R)})\right)^3$. Therefore the Aubin-Lions lemma can be applied, and up to a subsequence,
\begin{align}\label{L2convergence}
\nabla\Psi_\epsilon \rightarrow \nabla\Psi \quad \textnormal{in} \quad L^\infty\left([0,T];\left(L^2((0,R)\times{B_0(R)})\right)^3\right)
\end{align} 
We then diagonalize the subsequence to obtain strong convergence for any $R\in \mathbb{N}$.

Returning to the proof of \cref{main}, let $\nabla\Psi$ be the limit of $\nabla\Psi_\epsilon$ with convergence in the spaces specified in (\ref{itm:one})-(\ref{itm:three}) and \eqref{L2convergence}.  By (\ref{itm:one}) and integration by parts, 
\begin{align*}
\lim_{\epsilon\rightarrow 0} \int_0^T \int_0^\infty \int_{\mathbb{R}^2} \nabla \phi \cdot \nabla F_{\epsilon} \,dx\,dz\,dt &= \lim_{\epsilon\rightarrow 0} \bigg{(} -\int_0^T \int_0^\infty \int_{\mathbb{R}^2} \phi f_{L,\epsilon}\,dx\,dz\,dt\\
&\qquad\qquad + \int_0^T \int_{\mathbb{R}^2} \phi \left( f_{\nu,\epsilon}-\epsilon \half \dnu_\epsilon \right)\,dx\,dt \bigg{)} \\
&= \lim_{\epsilon\rightarrow 0} \left( -\int_0^T \int_0^\infty \int_{\mathbb{R}^2} \phi f_{L,\epsilon} + \int_0^T \int_{\mathbb{R}^2} \left( \phi  f_{\nu,\epsilon}-\epsilon \half \phi \dnu_{\epsilon,1} \right) \right) \\
&= -\int_0^T \int_0^\infty \int_{\mathbb{R}^2} \phi f_{L} + \int_0^T \int_{\mathbb{R}^2} \phi  f_{\nu} \\
&= \int_0^T \int_0^\infty \int_{\mathbb{R}^2} \nabla \phi \cdot \nabla F
\end{align*}
for $F$ solving the boundary value problem $\Delta F = f_L$ and $\partial_\nu F = f_\nu$.
Second, by \eqref{L2convergence},
$$ \lim_{\epsilon\rightarrow 0}\int_0^T \int_0^\infty \int_{\mathbb{R}^2} \left( \grad^\perp\Psi_\epsilon : \grad \nabla \phi \right) \cdot \nabla \Psi_\epsilon \,dx\,dz\,dt = \int_0^T \int_0^\infty \int_{\mathbb{R}^2} \left( \grad^\perp\Psi : \grad \nabla \phi \right) \cdot \nabla \Psi \,dx\,dz\,dt$$
In addition, it is immediate that
$$ \lim_{\epsilon\rightarrow 0} \int_0^\infty \int_{\mathbb{R}^2} \nabla\phi (0,z,x)\cdot \nabla\Psi_\epsilon(0,z,x)\,dx\,dz = \int_0^\infty \int_{\mathbb{R}^2} \nabla\phi(0,z,x)\cdot \nabla\Psi(0,z,x)\,dx\,dz$$
and
$$ \lim_{\epsilon\rightarrow 0} \int_0^T \int_0^\infty \int_{\mathbb{R}^2} \partial_t\nabla\phi\cdot \nabla\Psi_\epsilon\,dx\,dz\,dt = \int_0^T\int_0^\infty \int_{\mathbb{R}^2} \partial_t\nabla\phi\cdot \nabla\Psi\,dx\,dz\,dt. $$
Passing to the limit in \eqref{weakrqgsmoothdata}, we have that 
\begin{align*}
-\int_0^T \int_0^\infty &\int_{\mathbb{R}^2} \left(\left( \partial_t \nabla \phi+ \grad^\perp\Psi:\grad\nabla\phi \right) \cdot \nabla \Psi  + \nabla \phi \cdot \nabla F \right )\,dx\,dz\,dt \nonumber \\
 &= \int_0^\infty \int_{\mathbb{R}^2} \nabla\phi(0,z,x)\cdot\nabla\Psi(0,z,x)\,dx\,dz
\end{align*}
and thus $\Psi$ satisfies \cref{weakrqg}. The bound in the statement of the theorem follows from passing to the limit in $\epsilon$ in \cref{regularizedsolutions}(1)-(3), completing the proof. 
\end{proof}

\section{Proof of Theorem 1.2}
\begin{proof}[Proof of \cref{conservation}]
Define for all time
\begin{align*}
(\nabla \Psi ^\epsilon ) ^\epsilon(z,x) &:= \left(\nabla\Psi(z,\cdot) \ast \gamma_\epsilon\right) \ast \gamma_\epsilon (x)\\
&= \int_{\mathbb{R}^2} \int_{\mathbb{R}^2} \nabla \Psi(z,x-x'-\bar{x})\gamma_\epsilon(x')\gamma_\epsilon(\bar{x}) \,dx'\,d\bar{x}
\end{align*}
that is, we convolve $\nabla\Psi$ twice with a mollifier $\gamma_\epsilon$ in $x$ only, $z$ by $z$.  The extra mollification is for passage onto the nonlinear term later.  Strictly speaking, $(\nabla \Psi ^\epsilon ) ^\epsilon$ is not an admissible test function; it lacks compact support in space and time, and differentiability in $z$ and $t$.  However, let us proceed formally for the time being, and assume that $(\nabla\Psi^\epsilon)^\epsilon$ is admissible and that $\nabla\Psi$ is differentiable in time.  Multiplying $(rQG)$ by $(\nabla \Psi ^\epsilon ) ^\epsilon$ and integrating in space and from time $0$ to $t$, we obtain
\begin{align}\label{conservationid}
E_\epsilon(t) - E_\epsilon(0) &:=  \int_0^\infty \int_{\mathbb{R}^2}  \nabla \Psi ^\epsilon (t)  \cdot \nabla \Psi^\epsilon (t) \,dx \,dz -   \int_0^\infty \int_{\mathbb{R}^2}  \nabla \Psi ^\epsilon (0) \cdot \nabla \Psi^\epsilon (0) \,dx \,dz \nonumber\\
&= -2\int_0^t \int_0^\infty \int_{\mathbb{R}^2} \left( \grad^\perp\Psi : \grad (\nabla\Psi^\epsilon)^\epsilon \right) \cdot \nabla\Psi \,dx\,dz\,d\tau
\end{align}
We can now apply \cref{lpstuff}(\ref{dos}) to the right hand side to move the mollifier over, introduce the commutator between multiplication and mollification, and rewrite the nonlinear terms using tensor notation, obtaining
\begin{align*}
E_\epsilon(t) - E_\epsilon(0) &= -2 \int_0^t \int_0^\infty \int_{\mathbb{R}^2} \left\langle \left( \grad^\perp\Psi \otimes \nabla \Psi \right)^\epsilon - \left( \grad^\perp\Psi^\epsilon \otimes \nabla \Psi^\epsilon \right)  , \grad \nabla \Psi^\epsilon  \right\rangle \,dx \,dz \,d\tau\\
&\qquad- 2\int_0^t \int_0^\infty \int_{\mathbb{R}^2} \left\langle \grad^\perp\Psi^\epsilon \otimes \nabla \Psi^\epsilon, \grad \nabla \Psi^\epsilon  \right\rangle \,dx\,dz\,d\tau
\end{align*}
Integrating by parts in $x$ for fixed $z$ and $\tau$ gives that the second term is equal to zero. Applying \cref{lpstuff}(\ref{tres}) $z$ by $z$ with $f=\grad^\perp\Psi$ and $g=\nabla\Psi$ to the first term, we have
\begin{align*}
E_\epsilon(t) - E_\epsilon(0) &= -2 \int_0^t \int_0^\infty \iiint_{(\mathbb{R}^2)^3} \bigg{\langle} \left( \grad^\perp\Psi(x-\bar{x}) - \grad^\perp\Psi(x) \right)\otimes\\
&\left( \nabla\Psi(x-\bar{x}) - \nabla\Psi(x-x') \right), \grad\nabla\Psi(x) \bigg{\rangle}\cdot\gamma_\epsilon(\bar{x})\gamma_\epsilon(x') \,d\bar{x}\,dx'\,dx\,dz\,d\tau . \\  
\end{align*} 
Now apply H\"{o}lder's inequality in $x$ to obtain
\begin{align*}
|E_\epsilon(t) - E_\epsilon(0)| &\leq C \int_0^t\int_0^\infty \iint_{(\mathbb{R}^2)^2} \|\grad^\perp \Psi(z,\tau,\cdot-\bar{x})-\grad^\perp \Psi(z,\tau,\cdot) \|_{L^3(\mathbb{R}^2)} \gamma_\epsilon(\bar{x})\gamma_\epsilon(x')\\
&\quad \times \|\nabla \Psi(z,\tau,\cdot)-\nabla\Psi(z,\tau,\cdot-(x'-\bar{x})) \|_{L^3(\mathbb{R}^2)} \|\grad\nabla\Psi(z,\tau,\cdot) \|_{L^3(\mathbb{R}^2)}  \,d\bar{x}\,dx'\,dz\,d\tau.
\end{align*}
Integrating in $\bar{x}$ and $x'$, using the fact that $\gamma_\epsilon$ has integral one, and applying \cref{lpstuff}(\ref{cuatro}) $z$ by $z$ for $\bar{x},x'-\bar{x}< C\epsilon$ gives
\begin{align*}
|E_\epsilon(t) - E_\epsilon(0)|&\leq C \int_0^t\int_0^\infty \|\grad^\perp \Psi(z,\tau,\cdot) \|_{\Balpha(\mathbb{R}^2)} \|\nabla\Psi(z,\tau,\cdot) \|^2_{\Balpha(\mathbb{R}^2)} \epsilon^{3\alpha-1} \,dz\,d\tau \\
&\leq C \epsilon^{3\alpha -1} \|\nabla\Psi\|^3_{L^3 \left( [0,T)\times [0,\infty); \Balpha(\mathbb{R}^2) \right)}
\end{align*}
which approaches $0$ as $\epsilon\rightarrow 0$ if $\alpha>\frac{1}{3}$.  

We must now account for that fact that $(\nabla\Psi^\epsilon)^\epsilon$ is not an admissible test function.  Replacing $\Psi$ (which is a well-defined function in $C\left([0,T);L^6(\Rthreeplus)\right)$ by Sobolev embedding) with 
$$ \Psi_\eta := \left(\mathcal{X}_{\{|(x,z)| \leq \frac{1}{\eta}, \tau\leq T-\eta \}} \Psi \right) \ast \Gamma_\eta $$
for $\Gamma_\eta$ a space-time mollifier in $\mathbb{R}^4$ ensures compact support and differentiability in $z$ and $t$.   Then after mollifying as before in $x$, we can use $(\nabla (\Psi_\eta)^\epsilon)^\epsilon$ as a test function.  It is well known that \eqref{conservationid} holds when differentiability in time is replaced with $C\left([0,T);L^2(\Rthreeplus)\right)$. Passing to the limit in $\eta$ first and then in $\epsilon$ gives that
$$
\|\nabla\Psi(t)\|_{L^2(\Rthreeplus)}^2 -\|\nabla\Psi(0)\|_{L^2(\Rthreeplus)}^2 = \lim_{\epsilon\rightarrow 0}E_\epsilon(t)-E_\epsilon(0)  = 0,$$
completing the proof.
\end{proof}

\section{Proof of Theorem 1.3}
We divide up the proof into parts (1), (2) and (3).  
\begin{proof}[Proof of \cref{connectionstheorem}(1)]
The first step shows that integration by parts is valid for the reformulated equation, and the second step then integrates by parts to prove the claim.  

$Step$ $One:$ First, we extend the Sobolev function $\Psi$ to $\mathbb{R}^3$, denoting the extended function by $\Psi_E$. Let $\{ \Gamma_\epsilon \}_{\epsilon>0}$ be a sequence of approximate identities in $\mathbb{R}^3$.  Define 
$$ \nabla\Psi_{E,\epsilon} := \nabla\Psi_E \ast \Gamma_\epsilon$$
for $\epsilon>0$.  By assumption, we have 
$$ \Delta\Psi \in L^\infty\left([0,T);L^q(\Rthreeplus)\right) , \qquad \dnu \in L^\infty\left([0,T);L^p(\mathbb{R}^2)\right) $$
for $q\in[\frac{3}{2},3]$ and $p\in[2,\infty]$.  Combined with the elliptic estimates in \cref{laxmilgram1} and \cref{laxmilgram2}, this ensures that integration by parts for $\nabla\Psi_{E,\epsilon}$ is valid, and thus for $\phi$ compactly supported in $\Rthreeplus$ and time, 
\begin{align*}
-\int_0^T \int_0^\infty \int_{\mathbb{R}^2} &\left(\left( \partial_t \nabla \phi + \grad^\perp\Psi_{E,\epsilon} : \grad \nabla \phi \right)\cdot\nabla\Psi_{E,\epsilon} + \nabla\phi\cdot\nabla F \right) \,dx\,dz\,dt  \\
& \qquad =\int_0^T \int_0^\infty \int_{\mathbb{R}^2} \left(\left( \partial_t \phi + \grad^\perp \Psi_{E,\epsilon} \cdot \grad \phi \right) \Delta\Psi_{E,\epsilon}  + \phi \Delta F \right)  \,dx\,dz\,dt \nonumber\\
&\qquad\qquad - \int_0^T \int_{\mathbb{R}^2} \left(\left( \partial_t \phi + \grad^\perp \Psi_{E,\epsilon} \cdot \grad \phi \right) \dnu_{E,\epsilon} + \phi \partial_\nu F \right) \,dx\,dt\\
\end{align*}
and
\begin{align*}
\int_0^\infty \int_{\mathbb{R}^2} \nabla \phi(0,z,x) \cdot \nabla\Psi_\epsilon(0,z,x) \,dx\,dz &= -\int_0^\infty \int_{\mathbb{R}^2} \phi(0,z,x) \Delta \Psi_\epsilon(0,z,x) \,dx\,dz   \nonumber\\
& \qquad+  \int_{\mathbb{R}^2} \phi(0,0,x)\dnu_\epsilon(0,0,x) \,dx.
\end{align*}
We now argue that passing to the limit is justified in each identity.  We have that \cref{laxmilgram1}, \cref{laxmilgram2} give that $\nabla\Psi \in L^\frac{3q}{3-q}(\Rthreeplus) + L^\frac{3p}{2}(\Rthreeplus)$ for all time.  Noticing that since $q\geq \frac{3}{2}$ and $p\geq 2$, we have that
$$ \frac{3q}{3-q} \geq 3, \quad \frac{3p}{2} \geq 3, $$
and the following convergences follow:
$$ \Delta\Psi_{E,\epsilon} \rightarrow \Delta\Psi \quad \textnormal{in} \quad L^2\left([0,T);L_{loc}^\frac{3}{2}(\Rthreeplus)\right)$$
$$ \nabla\Psi_{E,\epsilon} \rightarrow \nabla\Psi \quad \textnormal{in} \quad L^2\left([0,T);L_{loc}^3(\Rthreeplus)\right).$$
Furthermore, using H\"{o}lder's inequality shows that for each fixed time, $\grad^\perp\Psi \Delta\Psi \in L^1_{loc}(\Rthreeplus)$ and $\grad^\perp\Psi \otimes \nabla\Psi \in L_{loc}^1(\Rthreeplus)$. Therefore,
$$ \grad^\perp\Psi_{E,\epsilon}\Delta\Psi_{E,\epsilon} \rightarrow \grad^\perp\Psi \Delta\Psi \quad \textnormal{in} \quad L^1\left( [0,T);L_{loc}^1(\Rthreeplus) \right) $$
$$ \grad^\perp\Psi_{E,\epsilon}\otimes\nabla\Psi_{E,\epsilon} \rightarrow \grad^\perp\Psi \otimes \nabla\Psi \quad \textnormal{in} \quad L^1\left([0,T);L_{loc}^1(\Rthreeplus)\right) $$
Finally, we have that
$$\frac{2q}{3-q} \geq 2,$$ 
and \cref{trace} gives $\nabla\Psi_2|_{z=0} \in L^\frac{2q}{3-q}(\mathbb{R}^2)$. Recalling that $\dnu\in L^2(\mathbb{R}^2)$ and $\grad^\perp\Psi_1 = -\riesz^\perp\dnu$, applying H\"{o}lder again gives $\grad^\perp\Psi \dnu \in L^1_{loc}(\mathbb{R}^2)$. It therefore follows that 
$$ \dnu_{E,\epsilon} \rightarrow \dnu \quad \textnormal{in} \quad L^2\left([0,T); L^2_{loc}(\mathbb{R}^2)\right) $$
$$ \grad^\perp \Psi_{E,\epsilon}|_{z=0} \rightarrow \grad^\perp \Psi|_{z=0}  \quad \textnormal{in} \quad L^2\left([0,T);L_{loc}^2(\mathbb{R}^2)  \right) $$
and
$$ \grad^\perp\Psi_{E,\epsilon} \partial_\nu \Psi_{E,\epsilon} \rightarrow \grad^\perp \Psi \dnu \quad \textnormal{in} \quad L^1\left([0,T); L_{loc}^1(\mathbb{R}^2)\right) $$
Letting $\epsilon$ tend to $0$ shows that 
\begin{align}\label{ibp1}
-\int_0^T \int_0^\infty \int_{\mathbb{R}^2} &\left(\left( \partial_t \nabla \phi + \grad^\perp\Psi : \grad \nabla \phi \right)\cdot\nabla\Psi + \nabla\phi\cdot\nabla F \right) \,dx\,dz\,dt \nonumber \\
& \qquad =\int_0^T \int_0^\infty \int_{\mathbb{R}^2} \left(\left( \partial_t \phi + \grad^\perp \Psi \cdot \grad \phi \right) \Delta\Psi  + \phi \Delta F \right)  \,dx\,dz\,dt \nonumber\\
&\qquad\qquad - \int_0^T \int_{\mathbb{R}^2} \left(\left( \partial_t \phi + \grad^\perp \Psi\cdot \grad \phi \right) \dnu + \phi \partial_\nu F \right) \,dx\,dt.
\end{align}
and
\begin{align}\label{ibp2}
\int_0^\infty \int_{\mathbb{R}^2} \nabla \phi(0,z,x) \cdot \nabla\Psi(0,z,x) \,dx\,dz &= -\int_0^\infty \int_{\mathbb{R}^2} \phi(0,z,x) \Delta \Psi(0,z,x) \,dx\,dz   \nonumber\\
& \qquad+  \int_{\mathbb{R}^2} \phi(0,0,x)\dnu(0,0,x) \,dx.
\end{align}

$Step$ $Two:$ Let us start by assuming that $\nabla\Psi$ satisfies \cref{weakrqg}. Then we have that 
\begin{align*}
-\int_0^T \int_0^\infty \int_{\mathbb{R}^2} &\left(\left( \partial_t \nabla \phi + \grad^\perp\Psi : \grad \nabla \phi \right)\cdot\nabla\Psi + \nabla\phi\cdot\nabla F \right) \,dx\,dz\,dt = \\
&\int_0^\infty \int_{\mathbb{R}^2} \nabla \phi(0,z,x) \cdot \nabla\Psi(0,z,x) \,dx\,dz,
\end{align*}
i.e. the left hand side of \eqref{ibp1} is equal to the left hand side of \eqref{ibp2}.  Choosing $\phi$ to be compactly supported in $[-T,T]\times \Rthreeplus$ gives that 
\begin{align*}
\int_0^T \int_0^\infty \int_{\mathbb{R}^2} &\left(\left( \partial_t \phi + \grad^\perp \Psi \cdot \grad \phi \right)\cdot \Delta\Psi  + \phi \Delta F \right)  \,dx\,dz\,dt   \\
&\qquad =-\int_0^\infty \int_{\mathbb{R}^2} \phi(0,z,x) \Delta \Psi(0,z,x) \,dx\,dz , 
\end{align*}
and therefore $\nabla\Psi$ satisfies \eqref{qg1}.  

To show that $\nabla\Psi$ satisfies \eqref{qg2}, choose $\bar{\phi}$ to be a test function compactly supported in $[-T,T] \times \mathbb{R}^2$. Let $\gamma(z)$ be a smooth function of one variable compactly supported in $[-1,1]$ with $\gamma\equiv 1$ for $x\in[-\frac{1}{2},\frac{1}{2}]$.  Let $\gamma_n (z) = \gamma(nz)$.  Define $\phi_{n}(t,z,x) = \gamma_n(z) \bar{\phi}(t,x)$.  Then $\grad\phi_{n}$, $\partial_t \phi_{n}$, and $\phi_{n}$ converge to $0$ in $\Rthreeplus$ (both pointwise and in any Lebesgue space). We have that the right hand side of \eqref{ibp1} is equal to the right hand side of \eqref{ibp2}.  Then plugging in $\phi_n$ as a test function, letting $n$ tend to infinity, and passing to the limit shows that
\begin{align}\label{boundaryequality}
 - \int_0^T \int_{\mathbb{R}^2} &\left(\left( \partial_t \bar{\phi} + \grad^\perp\Psi \cdot \grad \bar{\phi} \right) \dnu + \bar{\phi} \partial_\nu F \right) \,dx\,dt \nonumber \\
&\qquad = \int_{\mathbb{R}^2} \bar{\phi}(0,0,x) \dnu(0,x) \,dx.
\end{align}

Now assume for the other direction that $\Psi$ verifies \cref{weakqg}. Then for $\phi$ compactly supported in $\Rthreeplus$ (and time) and $\bar{\phi}$ compactly supported in $\mathbb{R}^2$ (and time),
\begin{align}\label{compactsupport}
-\int_0^T \int_0^\infty &\int_{\mathbb{R}^2} \left(\left( \partial_t \phi + \grad^\perp \Psi \cdot \grad \phi  \right) \Delta \Psi + \phi f_L  \right) \,dx\,dz\,dt \nonumber \\
&= \int_0^\infty \int_{\mathbb{R}^2} \phi(0,z,x)\Delta\Psi(0,z,x) \,dx\,dz
\end{align}
and
\begin{align*}
-\int_0^T &\int_{\mathbb{R}^2} \left(\left(  \partial_t \bar{\phi} + \grad^\perp \Psi(t,0,x) \cdot \grad \bar{\phi}\right) \dnu(t,x) + \bar{\phi}f_\nu  \right) \,dx \,dt \nonumber \\
&= \int_{\mathbb{R}^2} \bar{\phi}(0,x)\dnu(0,x)\,dx .
\end{align*}
Before proceeding we show that \eqref{compactsupport} holds for $\phi$ compactly supported in $\mathbb{R}^3$ rather than $\Rthreeplus$.  Let $\phi$ be compactly supported in $\mathbb{R}^3$ and time.  Using $\gamma_n(z)$ as defined previously, define
$$\phi_{n}(t,z,x) = \left( 1-\gamma_n(z) \right) \phi(t,z,x). $$
Then $\phi_{n}$ is compactly supported in $\Rthreeplus$ and $\grad\phi_{n}$, $\partial_t\phi_{n}$, and $\phi_n$ converge to $\grad\phi$, $\partial_t \phi$, and $\phi$ respectively, both pointwise in $\Rthreeplus$ and in any Lebesgue space. Therefore
\begin{align*}
-\int_0^T \int_0^\infty &\int_{\mathbb{R}^2} \left(\left( \partial_t \phi + \grad^\perp \Psi \cdot \grad \phi  \right) \Delta \Psi + \phi f_L  \right) \,dx\,dz\,dt \\
&= \lim_{n\rightarrow \infty} -\int_0^T \int_0^\infty \int_{\mathbb{R}^2} \left(\left( \partial_t \phi_{n} + \grad^\perp \Psi \cdot \grad \phi_{n}  \right) \Delta \Psi + \phi_{n} f_L  \right) \,dx\,dz\,dt \\
&= \lim_{n\rightarrow \infty} \int_0^\infty \int_{\mathbb{R}^2} \phi_{n}(0,z,x)\Delta\Psi(0,z,x) \,dx\,dz\\
&= \int_0^\infty \int_{\mathbb{R}^2} \phi(0,z,x) \Delta\Psi(0,z,x)\,dx\,dz,.
\end{align*}

We have then that the right hand side of \eqref{ibp1} is equal to the right hand side of \eqref{ibp2}, showing then that the left hand side of \eqref{ibp1} is equal to the left hand side of \eqref{ibp2}.  Therefore, $\nabla\Psi$ satisfies \cref{weakrqg} and is a weak solution to $(rQG)$.

\end{proof}

\begin{proof}[Proof of \cref{connectionstheorem}(2)]
As in part (1), the proof is split up into two steps.  

$Step$ $One:$ We assume that $p\in(\frac{4}{3},2]$, $q\in[\frac{3}{2},3]$, and 
$$ p\geq\frac{2q}{3(q-1)}. $$  Let us first point out the implications of the assumptions on $p$ and $q$. Throughout, we use the definitions of $\Psi_1$ and $\Psi_2$ described in \cref{weakqgc}. First, since $q\geq\frac{3}{2}$, \cref{laxmilgram1} ensures that for all time, $\nabla\Psi_2\in L^3_{loc}(\Rthreeplus)$, and therefore $\nabla\Psi_2\Delta\Psi \in L_{loc}^1(\Rthreeplus)$ is well-defined by H\"{o}lder's inequality.  Secondly, from \cref{laxmilgram2}, we have $\nabla\Psi_1 \in L^\frac{3p}{2}(\Rthreeplus)$.  Thus, the assumption that $p\geq\frac{2q}{3(q-1)}$ ensures that 
$$ \frac{2}{3p} \leq \frac{q-1}{q},$$ and therefore $\nabla\Psi_1 \Delta\Psi \in L_{loc}^1(\Rthreeplus)$ is also well-defined by H\"{o}lder's inequality.  Next, applying \cref{trace} to $\Psi_2$ gives that $\nabla\Psi_2|_{z=0}\in L^\frac{2q}{3-q}(\mathbb{R}^2)$.  Using that $p\geq\frac{2q}{3(q-1)}$, it follows that 
$$ \frac{3-q}{2q} \leq \frac{p-1}{p} .$$
Therefore, $\grad^\perp \Psi_2|_{z=0} \dnu \in L_{loc}^1(\mathbb{R}^2)$ is also well-defined from H\"{o}lder's inequality.  Combined with the fact that $p > \frac{4}{3}$, we can apply \cref{Marchand}, yielding that 
$$ \left( \grad^\perp \Psi \dnu \right)_C $$
is well-defined as a distribution.  

The proof now proceeds as before.  We regularize and extend $\nabla\Psi$ to $\nabla\Psi_{E,\epsilon}$.  Then 
\begin{align*}
-\int_0^T \int_0^\infty \int_{\mathbb{R}^2} &\left(\left( \partial_t \nabla \phi + \grad^\perp\Psi_{E,\epsilon} : \grad \nabla \phi \right)\cdot\nabla\Psi_{E,\epsilon} + \nabla\phi\cdot\nabla F \right) \,dx\,dz\,dt  \\
& \qquad =\int_0^T \int_0^\infty \int_{\mathbb{R}^2} \left(\left( \partial_t \phi + \grad^\perp \Psi_{E,\epsilon} \cdot \grad \phi \right) \Delta\Psi_{E,\epsilon}  + \phi \Delta F \right)  \,dx\,dz\,dt \nonumber\\
&\qquad\qquad - \int_0^T \int_{\mathbb{R}^2} \left(\left( \partial_t \phi + \grad^\perp \Psi_{E,\epsilon} \cdot \grad \phi \right) \dnu_{E,\epsilon} + \phi \partial_\nu F \right) \,dx\,dt\\
& \qquad =\int_0^T \int_0^\infty \int_{\mathbb{R}^2} \left(\left( \partial_t \phi + \grad^\perp \Psi_{E,\epsilon} \cdot \grad \phi \right) \Delta\Psi_{E,\epsilon}  + \phi \Delta F \right)  \,dx\,dz\,dt \nonumber\\
&\qquad\qquad - \int_0^T \int_{\mathbb{R}^2} \left(\partial_t \phi \dnu_{E,\epsilon} + \left(\grad^\perp \Psi_{E,\epsilon} \dnu_{E,\epsilon}\right)_C \cdot \grad \phi + \phi \partial_\nu F \right) \,dx\,dt.
\end{align*}
The second equality holds since the smoothness of $\nabla\Psi_{E,\epsilon}$ ensures that $\left( \grad^\perp \Psi_{E,\epsilon} \dnu_{E,\epsilon} \right)_C $ as a distribution is equal to the regular distribution $\left( \grad^\perp \Psi_{E,\epsilon} \dnu_{E,\epsilon} \right)$ (see \cref{Marchand}).
We have
$$ \Delta\Psi_{E,\epsilon} \rightarrow \Delta\Psi \qquad \textnormal{in} \qquad L^2\left([0,T);L^q(\Rthreeplus)\right) $$
$$ \nabla\Psi_{E,\epsilon} \rightarrow \nabla\Psi \qquad \textnormal{in} \qquad L^2\left([0,T);L_{loc}^\frac{q}{q-1}(\Rthreeplus)\right) ,$$
and therefore
$$\nabla\Psi_{E,\epsilon} \Delta\Psi_{E,\epsilon} \rightarrow \nabla\Psi\Delta\Psi \qquad \textnormal{in} \qquad L^1\left([0,T);L_{loc}^1(\Rthreeplus)\right) $$
and
$$ \grad^\perp\Psi_{E,\epsilon} \otimes \nabla\Psi_{E,\epsilon} \rightarrow \grad^\perp\Psi\nabla\Psi \qquad \textnormal{in} \qquad L^1\left([0,T);L_{loc}^1(\Rthreeplus)\right) $$
We have 
$$ \dnu_{E,\epsilon} \rightarrow \dnu \qquad \textnormal{in} \qquad L^2\left([0,T);L^p(\mathbb{R}^2)\right) $$
and the weak-* convergence
$$ \dnu_{E,\epsilon} \rightarrow \dnu \qquad \textnormal{in} \qquad L^\infty\left([0,T);L^{p}(\mathbb{R}^2)\right). $$
In addition,
$$ \grad^\perp \Psi_{E,\epsilon, 2}|_{z=0} \rightarrow \grad^\perp \Psi_2|_{z=0} \qquad \textnormal{in} \qquad L^2\left([0,T);L_{loc}^\frac{p}{p-1}(\mathbb{R}^2)\right), $$
and applying \cref{Marchand} since $p>\frac{4}{3}$ yields that
\begin{align*}
\grad \cdot \left( \grad^\perp \Psi_{E,\epsilon} \dnu_{E,\epsilon} \right)_C &= \grad \cdot \left( \grad^\perp \Psi_{E,\epsilon, 2} \dnu_{E,\epsilon} - \dnu_{E,\epsilon} \riesz^\perp \dnu_{E,\epsilon} \right) \\
&\rightarrow \grad\cdot \left(  \grad^\perp\Psi \dnu \right)_C \qquad \textnormal{in} \qquad \mathcal{D}'\left((0,T)\times\mathbb{R}^2\right). 
\end{align*}
Passing to the limit shows that 
\begin{align}\label{ibp3}
-\int_0^T \int_0^\infty \int_{\mathbb{R}^2} &\left(\left( \partial_t \nabla \phi + \grad^\perp\Psi : \grad \nabla \phi \right)\cdot\nabla\Psi + \nabla\phi\cdot\nabla F \right) \,dx\,dz\,dt  \nonumber\\
& \qquad =\int_0^T \int_0^\infty \int_{\mathbb{R}^2} \left(\left( \partial_t \phi + \grad^\perp \Psi \cdot \grad \phi \right)\cdot \Delta\Psi  + \phi \Delta F \right)  \,dx\,dz\,dt \nonumber\\
&\qquad\qquad - \int_0^T \int_{\mathbb{R}^2} \left(\partial_t \phi \dnu + \left(\grad^\perp \Psi \dnu\right)_C \cdot \grad \phi + \phi \partial_\nu F \right) \,dx\,dt
\end{align}
and
\begin{align}\label{ibp4}
\int_0^\infty \int_{\mathbb{R}^2} \nabla \phi(0,z,x) \cdot \nabla\Psi(0,z,x) \,dx\,dz &= -\int_0^\infty \int_{\mathbb{R}^2} \phi(0,z,x) \Delta \Psi(0,z,x) \,dx\,dz   \nonumber\\
& \qquad+  \int_{\mathbb{R}^2} \phi(0,0,x)\dnu(0,0,x) \,dx,.
\end{align}

$Step$ $Two:$  Assuming \eqref{ibp3} and \eqref{ibp4} hold, we can argue precisely as in the proof of \cref{connectionstheorem}(1) to prove the theorem.  We refer the reader to the proof of \cref{connectionstheorem}(1) for further details. 
\end{proof}

\begin{proof}[Proof of \cref{connectionstheorem}(3)]
The claim follows immediately from the observation that since $\Psi_1$ is harmonic,
$$\grad^\perp \Psi_1 = (0, -\partial_{x_2}\Psi_1, \partial_{x_1}\Psi_1) = (0, \mathcal{R}_2(\partial_\nu\Psi_1), -\mathcal{R}_1(\partial_\nu\Psi_1)) = -\riesz^\perp (\dnu_1)$$
and the claim in \cref{Marchand} that
$$ \int_{\mathbb{R}^2} f \riesz^\perp f \cdot \grad \phi = -\frac{1}{2} \int_{\mathbb{R}^2} \left(\riesz^\perp f \right)\cdot \left( [\bar{\Lambda}, \grad\phi] \left(\bar{\Lambda}^{-1}f\right)\right)$$
for $f\in L^2(\mathbb{R}^2)$.
\end{proof}

\section{Appendix}
\begin{proof}[Proof of \cref{regularizedsolutions}]
The differences between the setting of \cite{novackvasseur} and $(QG)_\epsilon$ are the presence of smooth forcing terms and the replacement of the term $-\lap \Psi_\epsilon$ (which comes from the physical consideration of Ekman layers) with the simplified diffusive term $\half \dnu_\epsilon$.  When $\Delta\Psi\equiv 0$, the two diffusive terms are equal.  When $\Psi$ is not harmonic, $\half \dnu_\epsilon$ is easier to analyze, as it ignores the effect of interior vorticity which appears in the term $\lap \Psi_\epsilon$. Local in time existence of smooth solutions in both \cite{novackvasseur} and $(QG)_\epsilon$ follows from classical semigroup techniques, such as those formulated by Kato \cite{kato}. Then, the proof of global existence in \cite{novackvasseur} is predicated on estimates which show that a sufficient level of regularity of the trajectories of the velocity field $\grad^\perp \Psi_\epsilon$ depends \textit{only} on quantities which are preserved by the evolution of the system. Applying a continuation criterion finishes the proof. Our goal is to provide an outline of the simple changes needed to apply those arguments to $(QG)_\epsilon$.

In the context of the Euler equations, a threshold for propagation of regularity is given in the following slight generalization of the Beale-Kato-Majda criterion (see \cite{bcd} for example).  If the solution $v$ is log-Lipschitz ($LL$) in space for each time and satisfies
$$ \int_0^T \| v(t) \|_{LL} < \infty, $$
then any sufficiently smooth solution can be continued beyond time $T$. In the context of $(QG)$ and the regularized system $(QG)_\epsilon$, the velocity field is stratified, indicating that regularity in the flat variables only should be enough to propagate Sobolev regularity.  It is well known that the endpoint Besov space $\Bdot_{\infty,\infty}^1(\mathbb{R}^2)$ embeds in the space of log-Lipschitz functions $LL$ (we refer again to \cite{bcd} for a discussion of these endpoint spaces).  The content of section 5 in \cite{novackvasseur} is thus an adaptation of the Beale-Kato-Majda argument for (QG) which shows that if
\begin{align}\label{besovestimate}
\grad^\perp \Psi \in L^\infty([0,T]\times[0,\infty);\Bdot_{\infty,\infty}^1(\mathbb{R}^2))
\end{align}
then the higher Sobolev norms of $\nabla\Psi$ satisfy a differential inequality on $[0,T]$, showing that the solution can be continued beyond time $T$. Adding smooth forcing terms $f_{L,\epsilon},f_{\nu,\epsilon}$ to the right hand side will introduce terms depending on Sobolev norms of $f_{L,\epsilon}, f_{\nu,\epsilon}$ into the differential inequality for $\|\nabla\Psi\|_{H^s}$; as long as the forcing terms are smooth, the argument functions in the same manner as the non-forced case.

The bulk of the argument of \cite{novackvasseur} then consists of showing that the estimate \eqref{besovestimate} is preserved by the evolution of the system and does not blow up in finite time. This is achieved in three main steps.  First, the de Giorgi technique is applied to obtain a $C^\alpha$ estimate on $\dnu$.  Second, a bootstrapping argument combining potential theory and Littlewood-Paley techniques shows that $\dnu \in L^\infty([0,T];\Bdot_{\infty,\infty}^1(\mathbb{R}^2))$.  Third, it is shown that once  $\dnu \in L^\infty([0,T];\Bdot_{\infty,\infty}^1(\mathbb{R}^2))$, \eqref{besovestimate} must hold. The third step requires no adaptations and we briefly describe it now before moving to the first two.  Using the notation for $\Psi_{1,\epsilon}$ and $\Psi_{2,\epsilon}$ as throughout the paper, simple properties of the Riesz transform and the Poisson kernel (details are contained in the short discussion immediately preceeding Theorem 4.3 in \cite{novackvasseur}) show that 
$$ \|\grad^\perp \Psi_{1,\epsilon}\|_{L^\infty([0,T]\times[0,\infty);\Bdot_{\infty,\infty}^1)} \lesssim \|\dnu_{1,\epsilon}\|_{L^\infty([0,T];\Bdot_{\infty,\infty}^1)}=\|\dnu_\epsilon\|_{L^\infty([0,T];\Bdot_{\infty,\infty}^1)}. $$
In addition, since $\Delta \Psi_\epsilon = \Delta \Psi_{2,\epsilon}$ solves a transport equation with divergence free drift and smooth forcing, the method of characteristics shows that the $L^\infty(\Rthreeplus)$ norm of $\Delta \Psi_{2,\epsilon}$ depends only on the initial data and the forcing term $f_{L,\epsilon}$. Then, properties of the Riesz transforms and classical trace estimates for Besov spaces show that  $\grad^\perp\Psi_{2,\epsilon} \in L^\infty([0,T]\times[0,\infty);\Bdot_{\infty,\infty}^1(\mathbb{R}^2))$, and thus \eqref{besovestimate} holds.  For further details of this step we refer to Theorem 4.3 and the preceding discussion in \cite{novackvasseur}. We move then to the first two steps.

The de Giorgi argument is contained in section 3 of \cite{novackvasseur} and is written for equations with divergence free drift $u$ and forcing $f$
$$ \partial_t \theta + u \cdot \grad \theta + \half \theta = f. $$
Setting $\theta=\dnu_\epsilon$, $u=\grad^\perp \Psi_\epsilon$, and $f=f_{\nu,\epsilon}$ shows that $(QG)_\epsilon$ falls into this regime.  The steps of the De Giorgi argument include a global $L^\infty$ bound (Lemma 3.3), a local $L^\infty$ bound (Lemma 3.4), an isoperimetric lemma (Lemma 3.5), and a decrease in oscillation (Lemmas 3.6 and 3.7).  Combining each step yields a H\"{o}lder modulus of continuity which depends only on $\|\Psi_{0,\epsilon}\|_{H^3(\Rthreeplus)}$ and certain norms of the forcing $f_{\nu,\epsilon}$. As the initial data and forcing have been regularized, these bounds are satisfied for $(QG)_\epsilon$. To give a flavor of the de Giorgi arguments, we state the global $L^\infty$ bound; the following steps can be stated entirely analogously to the lemmas from \cite{novackvasseur} referenced above.  
\begin{lemma}[Global $L^\infty$ bound]
For any $M>0$, there exists $L>0$ such that the following holds.  Let $\theta \in L^\infty([-2,0];H^\frac{5}{2}(\mathbb{R}^2))$ be a solution to 
$$\partial_t \theta + u \cdot \grad \theta + \half \theta =  f$$
with 
$$||\theta||_{L^\infty([-2,0];L^2(\mathbb{R}^2))} + ||(-\lap)^{-\frac{1}{4}} f||_{L^\infty([-2,0]; C^{\frac{1}{2}}(\mathbb{R}^2))}<M$$
and $\dive u=0$.  Then $\theta(t,x)\leq L$ for $(t,x)\in[-1,0]\times\mathbb{R}^2$.
\end{lemma}

Finally, let us describe the bootstrapping argument. The bootstrapping argument in \cite{novackvasseur} is built around the observation that the Poisson kernel is the fundamental solution to the equation
$$ \partial_t \theta + \half \theta = 0. $$
The choice of $\half \dnu_\epsilon$ as the diffusive term ensures that $(QG)_\epsilon$ again falls into this regime. In both \cite{novackvasseur} and $(QG)_\epsilon$, two forcing terms appear on the right hand side. The first term in both settings is of the form $u \cdot \grad \theta$ and comes from the nonlinearity.  The second term in \cite{novackvasseur} comes from the effect of $\Psi_2$ on the diffusive term $\lap \Psi$, whereas in $(QG)_\epsilon$ it comes from $f_{\nu,\epsilon}$. Lemma 4.1 in \cite{novackvasseur} asserts that the regularity of the nonlinear term is effectively \textit{additive}. That is, if $f$ and $g$ are $C^\alpha$ and $\theta_1$ solves
$$ \partial_t \theta_1 + \half \theta_1 = \grad\cdot(g_1 g_2) , $$
the representation formula for $\theta_1$ given by the Poisson kernel gives a $C^{2\alpha}$ estimate on $\theta_1$. Repeating this argument bootstraps the regularity of $\theta_1$ all the way to $C^{1,\alpha}$ for any $\alpha\in(0,1)$. Setting $g_1=\grad^\perp\Psi_\epsilon$ and $g_2=\dnu_\epsilon$ allows us to apply Lemma 4.1 to $(QG)_\epsilon$. Finally, as the forcing term $f_{\nu,\epsilon}$ is smooth in space, the solution to the fractional heat equation 
$$ \partial_t \theta_2 + \half \theta_2 = f_{\nu,\epsilon} $$ is smooth in space as well. Setting $\theta=\theta_1 + \theta_2$ and combining the two arguments shows that $\theta \in L^\infty([0,T];\Bdot_{\infty,\infty}^1(\mathbb{R}^2))$.  We refer again to the discussion in \cite{novackvasseur} which precedes Theorem 4.3 for details on the combination of these two steps.  The conclusion is that $  \dnu_\epsilon \in {L^\infty([0,T];\Bdot_{\infty,\infty}^1)} $ with norm depending only on the initial data $\Psi_{0,\epsilon}$ and the forcing terms, concluding the construction of a smooth solution.

The estimate in (1) follows from the method of characteristics.  For (2), recall that for $0<\alpha \leq 2$, $1\leq p < \infty$ the following inequality holds (see \cite{Marchand} for example):
$$0 \leq \int_{\mathbb{R}^2} \theta |\theta|^{p-2} \Lambda^{\alpha} \theta. $$
Multiplying by $\dnu_\epsilon |\dnu_\epsilon|^{p-2}$, integrating by parts, and applying the inequality with $\theta=\dnu_\epsilon$ and $\alpha=1$ then shows (2) for $p<\infty$. The estimate for $p=\infty$ follows after noticing that the initial data $\dnu_{0,\epsilon}$ and the forcing $f_{\nu,\epsilon}$ are smooth and compactly supported (in space), allowing us to take the limit of the estimate as $p\rightarrow \infty$.  Applying \cref{laxmilgram1} and \cref{laxmilgram2} gives (3).
\end{proof}

\section{acknowledgements}
The author would like to thank Vlad Vicol for suggesting the problem of global existence for data in Lebesgue spaces and directing the author to the result of Marchand.    

\bibliography{references}
\bibliographystyle{plain}
\nocite{nguyen}
\nocite{kato}
\nocite{boundeddomains}
\nocite{cn}
\nocite{cn2}
\nocite{ci}
\nocite{ci2}
\nocite{pv}
\nocite{cmt}
\nocite{MR1312238}
\nocite{Marchand}
\nocite{Resnick}
\nocite{dg}
\nocite{bb}
\nocite{novackvasseur}
\nocite{Esc}
\nocite{bcd}
\nocite{constantin1994}
\nocite{Isett2015}
\nocite{2016arXiv161000676B}
\nocite{Temam2001}
\nocite{a}
\nocite{arbogastandbona}
\nocite{2013arXiv1310.7947I}
\end{document}